\documentclass[11pt,reqno,a4paper]{amsart}

\usepackage{amsmath,amssymb,amsthm}
\usepackage{amsaddr}
\usepackage{mathtools}
\usepackage{enumitem}
\usepackage[margin=3cm,top=4cm,bottom=4cm,footskip=1cm,headsep=2cm]{geometry}
\usepackage[utf8]{inputenc}

\usepackage{amsthm}
\usepackage[british]{babel}

\usepackage{color}
\usepackage{xcolor}

\let\voo\v

\DeclarePairedDelimiter{\norm}{\|}{\|}
\DeclarePairedDelimiter{\snorm}{|}{|}

\def\v{\mathbf{v}}
\def\be{\mathbf{e}}

\def\la{\langle}
\def\ra{\rangle}

\def\div{\operatorname{div}}

\def\RR{\mathbb{R}}

\def\Itau{\mathcal{I}_\tau}
\def\I{I}
\def\Th{\mathcal{T}_h}
\def\Vh{\mathcal{V}_h}
\def\Qh{\mathcal{Q}_h}

\def\Zh{\mathcal{Z}_h}
\def\dt{\partial_t}
\def\dtt{\partial_{tt}}

\def\uu{\mathbf{u}}
\def\bbf{\mathbf{f}}
\def\vv{\mathbf{v}}

\def\rr{\mathbf{r}}

\def\D{\mathcal{D}}
\def\E{\mathcal{E}}
\def\Ah{\mathbf{A}_h}

\def\P{\mathbf{P}}

\def\softd{{\leavevmode\setbox1=\hbox{d}%
		\hbox to 1.05\wd1{d\kern-0.4ex{\char039}\hss}}}

\def\skw{{\operatorname{skw}}}

\def\bmu{\bar{\mu}}
\def\hbmu{\hat{\bar{\mu}}}
\def\bu{\bar{\uu}}
\def\hbu{\hat{\bar{\uu}}}
\def\hbp{\hat{\bar{p}}}
\def\hbphi{\hat{\bar{\phi}}}
\newcommand{\oobar}[1]{\mkern 1.5mu\overline{\mkern-1.5mu#1\mkern-1.5mu}\mkern 1.5mu}

\newtheorem{lemma}{Lemma}
\newtheorem{problem}[lemma]{Problem}
\newtheorem{theorem}[lemma]{Theorem}
\theoremstyle{definition}
\newtheorem{remark}[lemma]{Remark}

\author{A. Brunk$^1$, H. Egger$^{2,3}$, O. Habrich$^3$, and  
M. Luk\'a\voo{c}ov\'a-Medvi\softd ov\'a$^1$}

\title[Structure-preserving discretization for the CHNS system]{A second-order fully-balanced  structure-preserving variational\\ discretization scheme for the Cahn-Hilliard Navier-Stokes system}

\address{$^1$Institute of Mathematics, Johannes Gutenberg-University Mainz, Germany}
\address{$^2$Johann Radon Institute for Computational and Applied Mathematics, Linz, Austria}
\address{$^3$Institute for Numerical Mathematics, Johannes Kepler University Linz, Austria}

\email{abrunk@uni-mainz.de} \email{herbert.egger@jku.at} \email{oliver.habrich@jku.at} \email{lukacova@uni-mainz.de}

\begin{document}

\begingroup
\def\uppercasenonmath#1{} 
\let\MakeUppercase\relax 
\maketitle
\endgroup

\begin{abstract}
We propose and analyze a structure-preserving space-time variational discretization method for the Cahn-Hilliard-Navier-Stokes system. Uniqueness and stability for the discrete problem is established in the presence of concentration dependent mobility and viscosity parameters by means of the relative energy estimates and order optimal convergence rates are established for all variables using balanced approximation spaces and relaxed  regularity conditions on the solution.
Numerical tests are presented to demonstrate the proposed method is fully practical and yields the predicted convergence rates. The discrete stability estimates developed in this paper may also be used to analyse other discretization schemes, which is briefly outlined in the discussion. 
\end{abstract}

\section{Introduction} \label{sec:1}

The Cahn-Hilliard-Navier-Stokes system is frequently used as a diffuse interface model for the flow of two immiscible incompressible fluids and for modelling phase separation of a binary fluid after a deep quench; see \cite{abels2007,anderson1998,Hohenberg} for applications and references. %
In this paper, we study a system with constant mass densities, but concentration dependent mobility and viscosity, given by
\begin{alignat}{5} 
\dt \phi + \uu\cdot\nabla\phi&= \div(b(\phi) \nabla \mu), \label{eq:chns1}\\
\mu &= -\gamma \Delta \phi + f'(\phi), \label{eq:chns2} \\
\dt \uu + (\uu\cdot\nabla)\uu &= \div(\eta(\phi)\nabla\uu) - \nabla p - \phi\nabla\mu,
\label{eq:chns3} \\
\div(\uu)&=0.
\label{eq:chns4}
\end{alignat}
Here $\phi$ is the phase fraction or relative concentration of the two fluids, $\mu$ is the chemical potential, $\uu$ the velocity of the fluid mixture, $p$ the pressure, and $b(\cdot)$, $\eta(\cdot)$ are the mobility and viscosity functions.
Furthermore, $\gamma>0$ is the interface parameter and $f(\phi)$ a double well potential whose minima characterize the equilibrium concentrations of the two phases.
The density $\rho$ of the fluids is assumed constant, and set to one for simplicity.

The system~\eqref{eq:chns1}--\eqref{eq:chns4} describes a relaxation process towards equilibrium, expressed by decay of the free energy
$ \E(\phi,\uu) = \int_\Omega \frac{\gamma}{2} |\nabla \phi|^2 + f(\phi) + \frac{1}{2} |\uu|^2 \, dx$, 
more precisely
\begin{align} \label{eq:energydissipation}
\E(\phi(t),\uu(t)) + \int_s^t \D_{\phi}(\mu,\uu) \, dr
= \E(\phi(s),\uu(s)) , \qquad t>s, 
\end{align}
with $\D_\phi(\mu,\uu)=\int_\Omega b(\phi) |\nabla \mu|^2 + \eta(\phi) |\nabla \uu|^2 \, dx$ denoting the dissipation rate. 
This relation not only implies thermodynamic consistency of the model, but also allows to establish a-priori bounds and the existence of solutions. 
We refer to \cite{Abels2008,Boyer13} for a comprehensive analysis, including existence, uniqueness, and regularity results. 
Extension of the above mentioned model to logarithmic and more singular potentials with degenerate mobility are given in \cite{Boyer13}. An extended model including phase-dependent densities is derived in \cite{AGG}, whose analysis and error estimates can be found in \cite{Grn2013,GARCKE2016151}. \\

\textbf{Related results.}
Significant work has been attributed also to the numerical solution of \eqref{eq:chns1}--\eqref{eq:chns4}.
In \cite{Feng06}, a fully discrete scheme based on a finite element approximation in space combined with an implicit time-stepping scheme has been proposed and its convergence to weak solutions of the phasefield models and its sharp interface limit has been established. 
Great effort has further been devoted to reduce the computational cost of simulation by various splitting and linearization approaches. 
In \cite{ChenZhao2020,Chen2016,ZhaoHan2021}, finite difference methods in space are used together with first and second order linear implicit time discretization methods. \cite{Gong2018} considers an IEQ approach for dealing with the nonlinearity in the potential.
In \cite{Bao2012,Han2015}, mixed finite elements methods are employed with convex-concave splitting of the potential, leading again to efficient linear implicit time integration schemes. 
%
%
All these works formulate linear-implicit unconditionally stable time-stepping schemes of first or second order accuracy, complemented by appropriate spatial discretization approaches, in order to obtain efficient fully discrete methods. 
%

Much fewer results are available up to date concerning rigorous convergence estimates for the proposed discretization methods: 
Kay~et~al.~\cite{Kay2008,Kay2007} prove $O(h)$ convergence for a finite-element approximation of the Cahn-Hilliard Navier-Stokes problem with constant mobility; a fully-discrete scheme is presented for the numerical tests. 
Cai and Shen~\cite{CaiShen18} show $O(\tau + h^q)$ convergence for a  finite element discretization combined with a linear implicit and weakly coupled energy stable scheme in time. 
Liu and Riviere~\cite{LiuRiviere2020} prove $O(\tau + h^q)$ 
convergence estimates for a discontinuous Galerkin approximation combined with an implicit time-stepping scheme. 
$O(\tau^2+h^q)$ error estimates are established by Diegel et al.~\cite{Diegel2017} for a linear implicit time-stepping scheme combined with inf-sup stable finite elements for the spatial approximation. 
Li and Shen~\cite{LiShen2020} establish $O(\tau^2 + h^2)$ convergence for a linear implicit method using an auxiliary variable approach for the time discretization and a staggered-grid finite difference method for spatial approximation of the fluid flow. 
In the recent paper~\cite{LiShen2022}, the same authors prove $O(\tau)$ convergence rates for a fully decoupled time-stepping scheme. 
Let us emphasize that all the above results concerning convergence rates are obtained for constant mobility and viscosity only. \\

\textbf{Main contributions and possible extensions.}
We propose and analyze a fully discrete approximation scheme for a system with nonlinear mobility and viscosity parameters, that provably preserves the variational structure of the underlying equations exactly also on the discrete level.
Finite element spaces with balanced approximation properties are used for the spatial discretization, combined with a variational time-discretization scheme, akin to the average vector field method \cite{Gonzales96,McLachlanEtAl99}. 
The existence of discrete solutions can be shown without restriction on the discretization parameters. While uniqueness can be established under a mild CFL condition.
Order optimal error estimates for this fully-balanced approximation scheme are established in the presence of nonlinear model parameters for all solution components, including the pressure, and under weaker regularity assumptions on the exact solution compared to previous work. 
A main advantage of the the proposed scheme is that it automatically inherits the energy-dissipative structure of the continuous problem. This allows us to establish discrete stability by means of a discrete relative energy estimates, which greatly simplifies the error analysis.
These estimates also allows us to establish uniqueness of the discrete solution and stability with respect to perturbations measured in appropriate norms. We refer a reader to \cite{Juengel16}, \cite{LSY} for an introduction and further related results. 

Our results are presented in detail for a fully balanced second order approximation in space and time, but the method and its analysis generalize quite naturally to higher order approximations. 
As illustrated in numerical tests, the proposed scheme is fully practical, but a coupled nonlinear system has to be solved in every time step, which typically requires a few Newton iterations. 
As briefly discussed at the end of the paper, alternative, possibly more efficient, discretization schemes may be considered as perturbations of the proposed scheme, and the discrete stability estimates derived in this paper can be used to obtain corresponding error estimates.
The focus of the paper therefore lies on the nonlinear stability and error analysis of the proposed structure-preserving discretization scheme.

\bigskip 

\textbf{Outline.}
The remainder of the manuscript is organized as follows: 
Section~\ref{sec:prelim} presents our notation and basic assumptions,
and the basic ingredients for our discretization strategy.
In Section~\ref{sec:main}, we then introduce our numerical method and state our main theoretical results. 
Sections~\ref{sec:proof:thm:well-posed}--\ref{sec:proof:thm:fulldisk_part2} discuss the essential steps for the proofs of our theorems. 
For illustration of or theoretical results, we present some numerical tests in Section~\ref{sec:num}, and the main part of the paper closes with a short discussion.
Technical proof of auxiliary results, which are required for our analysis, are presented in Appendix~\ref{app:proj}--\ref{app:time-derivative}.
%

\section{Preliminaries and notation} \label{sec:prelim}

Before we present out discretization method and main results in detail, let us briefly introduce our notation and main assumptions, and recall some basic facts.

\subsection*{Notation}

The system \eqref{eq:chns1}--\eqref{eq:chns4} is investigated on a finite time interval $(0,T)$. 
To avoid the introduction of boundary conditions, we consider a spatially periodic setting, i.e., 
\begin{itemize}
\item[(A0)] $\Omega \subset \RR^d$, $d=2,3$ is a cube and identified with the $d$-dimensional torus $\mathcal{T}^d$.\\
Moreover, functions on $\Omega$ are assumed to be periodic throughout the paper. 
\end{itemize}
With minor changes, all results derived in the paper carry over to bounded domains with more general boundary conditions; see e.g. \cite{Boyer13,Diegel2017}.
By $L^p(\Omega)$, $W^{k,p}(\Omega)$, we denote the usual Lebesgue and Sobolev spaces with norms $\norm{\cdot}_{L^{p}}$ and $\norm{\cdot}_{W^{k,p}}$, respectively. 
As usual, we abbreviate $H^k(\Omega)=W^{k,2}(\Omega)$ and $\norm{\cdot}_{H^{k}} = \norm{\cdot}_{W^{k,2}}$. 
The corresponding dual spaces are denoted by $H^{-k}(\Omega)=H^k(\Omega)^*$, with norm defined by
\begin{align} \label{eq:dualnorm}
    \norm{r}_{-k} = \sup_{v \in H^s(\Omega)} \frac{\la r, v\ra}{\|v\|_{k}}.
\end{align}
Here $\langle \cdot, \cdot\rangle$ denotes the duality product on $H^{-k}(\Omega) \times H^k(\Omega)$. 
The same symbol is also used for the scalar product on $L^2(\Omega)$, which is defined by
\begin{align} \label{eq:sp}
\la u, v \ra = \int_\Omega u \cdot v \, dx \qquad \forall u,v \in L^2(\Omega).    
\end{align}
%
%
By $L^2_0(\Omega) \subset L^2(\Omega)$, we denote the spaces of square integrable functions with zero average.
%
%
As usual, we denote by $L^p(a,b;X)$, $W^{k,p}(a,b;X)$, and
$H^k(a,b;X)$, the Bochner spaces of integrable or differentiable functions on the time interval $(a,b)$ with values in some Banach space $X$. If $(a,b)=(0,T)$, we omit reference to the time interval and briefly write $L^p(X)$. The corresponding norms are denoted, e.g., by $\|\cdot\|_{L^p(X)}$ or $\|\cdot\|_{H^k(X)}$.

\subsection*{Main assumptions}

Throughout the paper, we utilize that the model parameters are sufficiently smooth and satisfy some typical assumptions. In particular, we require that 
\begin{itemize}\itemsep1ex
\item[(A1)] $\gamma>0$ is a positive constant;
\item[(A2)] $b\in C^2(\mathbb{R})$ with $0< b_1\leq b(s) \leq b_2$, $\norm{b'}_{L^\infty(\RR)}\leq b_3$, $\norm{b''}_{L^\infty(\RR)}\leq b_4$;
\item[(A3)] $f\in C^4(\mathbb{R})$ such that $f(s) \ge 0$ and $f''(s) \geq -f_1$, for some $f_1\geq0$. Furthermore, 
$$
|f^{(k)}(s)| \le f_2^{(k)} + f_3^{(k)} |s|^{4-k}, \qquad 0 \le k \le 4;
$$
\item[(A4)] $\eta\in C^2(\mathbb{R})$ with $0< \eta_1\leq \eta(s) \leq \eta_2$, $\norm{\eta'}_{L^\infty(\RR)}\leq \eta_3$, $\norm{\eta''}_{L^\infty(\RR)}\leq \eta_4$.
\end{itemize}
Existence and uniqueness of a regular periodic solution to \eqref{eq:chns1}--\eqref{eq:chns4} on $\Omega \times (0,T)$ can then be established for appropriate initial conditions $\phi(0)$, $\uu(0)$, and a restriction on the maximal time $T$ in dimension $d=3$; see \cite{Boyer13} for corresponding results.

\subsection*{Variational characterization}

Using standard arguments, one can see that any smooth periodic solution of \eqref{eq:chns1}--\eqref{eq:chns4} satisfies the variational identities
\begin{align}
\la \dt\phi,\psi \ra &- \la \phi\uu,\nabla\psi\ra + \la b(\phi)\nabla\mu,\nabla\psi \ra = 0, \label{eq:weak1} \\
\la \mu,\xi \ra &- \gamma\la \nabla\phi,\nabla\xi \ra - \la f'(\phi),\xi \ra = 0, \label{eq:weak2}\\
\la \dt\uu,\vv \ra &+ \la \uu \cdot \nabla \uu, \vv\ra_\skw + \la \eta(\phi)\nabla\uu,\nabla\vv \ra - \la p,\div\vv \ra + \la \phi\nabla\mu,\vv\ra=0, \label{eq:weak3} \\
\la \div\uu,q \ra & = 0, \label{eq:weak4}
\end{align}
for all sufficiently smooth periodic test functions $\psi$, $\xi$, $\vv$, $q$, and time $0 \le t \le T$.
As before, $\langle a,b\rangle=\int_\Omega a \, b \, dx$ denotes the usual scalar product, while the symbol
\begin{align} \label{eq:skw}
\la a \cdot \nabla b, c \ra_\skw := \frac{1}{2} (\la a \cdot \nabla b, c\ra - \la b,a\cdot \nabla b\ra)
\end{align}
is used to represent the skew-symmetric part of the convection term.
Further note that the components of the solution depend on time $t$, while the test functions are independent of $t$.
Hence the variational identities \eqref{eq:weak1}--\eqref{eq:weak4} have to be understood pointwise in time. 

\subsection*{Basic properties}

The variational characterization \eqref{eq:weak1}--\eqref{eq:weak4} immediately implies some important properties of solutions:
By testing with $\psi=1$, $\xi=0$, $\vv=0$, and $q=0$, we get 
\begin{align} \label{eq:mass}
\frac{d}{dt} \int_\Omega \phi(t) \, dx = 0,
\end{align}
which encodes the conservation of mass. 
Testing with $\psi=\mu(t)$, $\xi=\dt\phi(t)$, $\vv=\uu(t)$ and $q=p(t)$, on the other hand, directly leads to 
\begin{align} \label{eq:energy}
\frac{d}{dt} \E(\phi(t),\uu(t)) = - \D_{\phi(t)}(\mu(t),\uu(t)).
\end{align}
By integration in time, one then immediately recovers the energy--dissipation law \eqref{eq:energydissipation}. 
These important properties of solutions are thus encoded already in the particular variational characterization \eqref{eq:weak1}--\eqref{eq:weak4} of solutions.

\begin{remark} 
The energy identity \eqref{eq:energy} is based on skew-symmetry of various terms in the variational formulation, which can be preserved systematically on the discrete level. 
Similar  skew-symmetric variants of the convective term $\la \cdot, \ra_\skw$ in \eqref{eq:weak3} were also used in the works \cite{CaiShen18,ChenZhao2020,Feng06}, for instance. 
In contrast to the approach of \cite{Diegel2017}, the particular form of the convective term in \eqref{eq:weak1} allows us to guarantee conservation of mass on the discrete level also for balanced approximation spaces. 
\end{remark}

\subsection*{Space discretization}

As another preparatory step, let us introduce the relevant notation and assumptions for our discretization strategy. We require that
\begin{itemize}\itemsep1ex
    \item[(A5)] $\Th$ is a geometrically conforming and quasi-uniform partition of $\Omega$ into simplices that can be extended periodically to periodic extensions of $\Omega$. 
\end{itemize}
By quasi-uniform, we mean that there exists a constant $\sigma>0$ such that $\sigma h \le \rho_K \le h_K \le h$ for all $K \in \Th$, where $\rho_K$ and $h_K$ are the inner-circle radius and diameter of the element $K \in \Th$ and $h=\max_{K \in \Th} h_T$ is the global mesh size \cite{BrennerScott,John2016}. 
%
%
We denote by 
\begin{align*}
    \Vh &:= \{v \in H^1(\Omega) : v|_K \in P_2(K) \quad \forall K \in \Th\},\\
    \Qh &:= \{v \in L^2_0(\Omega)\cap H^1(\Omega) : v|_K \in P_1(K) \quad \forall K \in \Th\}
\end{align*}
the spaces of continuous and piecewise quadratic, respectively piecewise linear functions over $\Th$.
Elements of $\Qh$ further satisfy the zero average condition. 
Furthermore,
\begin{align} \label{eq:Zh}
\Zh = \{\vv_h \in \Vh^d : \la \div \vv_h,q_h\ra=0 \ \forall q_h \in \Qh\}    
\end{align}
stands for the space of discrete divergence free functions which play an important role in the discretization error analysis of flow problems.

By $\pi_h^0 : L^2(\Omega) \to \Vh$ and $\pi_h^1 : H^1(\Omega) \to \Vh$, we denote the $L^2$- and $H^1$-orthogonal projection operators, respectively, defined by 
\begin{alignat}{2}
\la \pi_h^0 u - u, v_h\ra &= 0 \qquad && \forall v_h \in \Vh, \label{eq:defl2proj}\\
\la \pi_h^1 u - u, v_h\ra + \la \nabla (\pi_h^1 u - u ), \nabla v_h \ra &= 0 \qquad && \forall v_h \in \Vh.\label{eq:defh1proj}
\end{alignat}
%
%
Similar notation is used for the projections of vector valued functions.
We further define the $L^2$-orthogonal projection $\tilde \pi_h^0 : L_0^2(\Omega) \to \Qh$ for the pressure space by
\begin{align}
    \la \tilde \pi_h^0 p - p, q_h\ra &= 0 \qquad \forall q_h \in \Qh.
\end{align}
We will further make use of the modified Stokes projector $\P^1_h : H^1(\Omega)^d \to \Zh \subset \Vh^d$, which is defined by the variational problem
\begin{alignat}{2}
\la \uu - \P_h^1 \uu,\nabla\vv_h \ra + \la \nabla (\uu - \P_h^1 \uu), \nabla \vv_h\ra &= 0 \qquad && \forall \vv_h \in \Zh. \label{eq:defnmodstokes}
\end{alignat}
Let us emphasize that $\P_h^1 \uu \in \Zh$ is discrete divergence free and that only such functions are used as test functions above. 
For convenience of the reader, some well-known approximation properties of these projection operators are summarized in Appendix~\ref{app:proj}.

\subsection*{Time discretization}

We also employ piecewise polynomial functions for the approximation in time. 
For ease of presentation, we consider a uniform grid 
\begin{itemize}\itemsep1ex
\item[(A6)] $\Itau:=\{0=t^0,t^1,\ldots,t^N=T\}$ \quad with time steps $t^n = n \tau$ and $\tau = T/N$.
\end{itemize}
More general non-uniform grids could be treated with minor modifications.
By 
\begin{align}
P_k(\Itau;X) 
\qquad \text{and} \qquad 
P_k^c(\Itau;X) = P_k(\Itau;X) \cap C(0,T;X),
\end{align}
we denote the spaces of discontinuous, respectively, continuous piecewise polynomial functions of degree $\le k$ over the time grid $\Itau$, with values in some vector space $X$.
We utilize a bar symbol $\bar u$ to denote functions in $P_0(\Itau;X)$ that are piecewise constant in time.

In our analysis, we use the piecewise linear interpolation
$\I_\tau^1:H^1(0,T;X)\to P_1^c(\Itau;X)$ as well as the $L^2$-orthogonal projection
$\bar \pi_\tau^0 : L^2(0,T;X) \to \Pi_0(\Itau;X)$ to piecewise constants in time. 
Properties of these operators are again summarized in Appendix~\ref{app:proj}. 
For ease of presentation, we will frequently use 
$\bar u = \bar \pi_\tau^0 u$ also to abbreviate the piecewise constant projection in time of a function $u \in L^2(X)$. 
For abbreviation, we introduce 
\begin{align} \label{eq:abn}
    \la a, b\ra^n := \int_{t^{n-1}}^{t^n} \la a(s), b(s) \ra \,ds,
\end{align}
and we write $\la a \cdot \nabla b, c\ra_\skw^n = \int_{t^{n-1}}^{t^n} \la a \cdot \nabla b, c\ra_\skw \, ds$ for the skew-symmetric variant.


%

\section{Discretization method and main results} \label{sec:main}

We are now in the position to introduce our discretization method and state our main results. As approximation of the initial value problem for
\eqref{eq:chns1}--\eqref{eq:chns4}, we consider the following scheme, which is motivated by the variational characterization  \eqref{eq:weak1}--\eqref{eq:weak4} of solutions.

\begin{problem}\label{prob:full} 
Find $\phi_{h,\tau} \in P_1^c(\Itau,\Vh)$, $\bar\mu_{h,\tau} \in P_0(\Itau;\Vh)$,  $\uu_{h,\tau} \in P_1^c(\Itau;\Vh^d)$ and $\bar p_{h,\tau} \in P_0(\Itau;\Qh)$, such that 
$\uu_{h,\tau}(0) = \P_h^1 \uu(0)$, $\phi_{h,\tau}(0) = \pi_h^1 \phi(0)$,  
and 
\begin{align}
\la \dt\phi_{h,\tau}, \bar\psi_{h,\tau}\ra^n 
  &= \la\phi_{h,\tau}\bar\uu_{h,\tau},\nabla\bar\psi_{h,\tau}\ra^n - \la b(\bar\phi_{h,\tau})\nabla\bar\mu_{h,\tau}\nabla\bar\psi_{h,\tau}\ra^n, \label{eq:pg1}\\
\la \bar\mu_{h,\tau},\bar\xi_{h,\tau}  \ra^n & = \gamma\la \nabla\phi_{h,\tau},\nabla\bar\xi_{h,\tau}\ra^n + \la f'(\phi_{h,\tau}),\bar\xi_{h,\tau}\ra^n, \label{eq:pg2}\\
 \la \dt \uu_{h,\tau} , \bar\vv_{h,\tau} \ra^n
&= \la \uu_{h,\tau} \cdot \nabla \bar\uu_{h,\tau},\bar\vv_{h,\tau}\ra_\skw^n -\la \eta(\bar\phi_{h,\tau}) \nabla\bu_{h,\tau}, \nabla\bar\vv_{h,\tau} \rangle^n \label{eq:pg3}\\
&\qquad \qquad +\langle \bar p_{h,\tau}, \div \bar\vv_{h,\tau}\ra^n - \la \phi_{h,\tau}\bar\vv_{h,\tau},\nabla\bar\mu_{h,\tau}\ra^n, \notag  \\
\la \div \uu_{h,\tau}, \bar q_{h,\tau} \ra^n &= 0, \label{eq:pg4}
\end{align}
for all $\bar\psi_{h,\tau},\bar \xi_{h,\tau} \in P_0(\Itau;\Vh)$, $\bar\vv_{h,\tau} \in P_0(\Itau;\Vh^d)$, $\bar q_{h,\tau} \in P_0(\Itau;\Qh)$ and all $n \le N$.
\end{problem}
\noindent
Recall that functions with a bar symbol are piecewise constant in time, and that $\bar \phi_{h,\tau} = \bar \pi_\tau^0 \phi_{h,\tau}$ also denotes the $L^2$-orthogonal projection onto such functions.

\begin{remark}
The proposed scheme consists of a mixed finite element approximation in space together with a variational time-discretization approach, akin to the average vector field method \cite{Gonzales96,McLachlanEtAl99}.
Since the test functions are discontinuous in time, the system~\eqref{eq:pg1}--\eqref{eq:pg4} yields an implicit time-stepping scheme, similar to a Crank-Nicolson method; see \cite{Akrivis11,AndreevSchweitzer2014} for details. 
The finite element pair employed for the fluid flow problem amounts to the Taylor-Hood element and is discrete inf-sup stable; see \cite{GiraultRaviart,John2016}. 
Further note that the approximation order of the finite element spaces and the time discretization are fully balanced. 
As we will show below, the proposed method is in fact second order accurate in space and time, and in all variables. 
\end{remark}

\subsection*{Well-posedness}
As a first step of our analysis, let us state the well-posedness of the discrete problem and summarize some elementary properties of its solutions. 

\begin{theorem} \label{thm:well-posed}
Let (A0)--(A6) hold. Then, for any $\phi(0) \in H^1(\Omega)$, $\uu(0) \in H^1(\Omega)^d$, and $h,\tau>0$, Problem~\ref{prob:full} has at least one solution. 
Moreover, any solution of  \eqref{eq:pg1}--\eqref{eq:pg4} satisfies 
\begin{align}
\int_\Omega \phi_{h,\tau}(t^n) \,dx 
&= \int_\Omega \phi_{h,\tau}(t^m) \,dx 
\qquad \text{and} \qquad \qquad \qquad \qquad \qquad \label{eq:mass_disc}\\ \E(\phi_{h,\tau}(t^n),\uu_{h,\tau}(t^n)) &+ \int_{t^m}^{t^n} \D_{\bar\phi_{h,\tau}}(\bmu_{h,\tau},\bu_{h,\tau}) \, ds = \E(\phi_{h,\tau}(t^m),\uu_{h,\tau}(t^m))
\label{eq:energy_disc}
\end{align}
for all discrete time steps 
$t^m \le t^n \in \Itau$, with energy- and dissipation functionals 
\begin{align*}
\E(\phi,\uu)) = \int_\Omega \frac{\gamma}{2} |\nabla \phi|^2 + f(\phi) + \frac{1}{2} |\uu|^2 \, dx 
\quad \text{and} \quad 
\D_\phi(\mu,\uu) = \int_\Omega b(\phi) |\nabla \mu|^2 + \eta(\phi) |\nabla \uu|^2 \, dx.
\end{align*}
\end{theorem}
\noindent 
The main steps of the proof of this result are presented in Section~\ref{sec:proof:thm:well-posed}.
Let us emphasize that, by our variational discretization approach, the important properties \eqref{eq:mass} and \eqref{eq:energy} of solutions to the system \eqref{eq:chns1}--\eqref{eq:chns4} carry over verbatim to solutions of the discrete problem. 

\subsection*{Error estimates}
In order to obtain quantitative convergence rates, we additionally have to assume that the true solution is sufficiently smooth. 
For our analysis, we require
\begin{itemize}
\item[(A7)] a sufficiently regular solution of \eqref{eq:chns1}--\eqref{eq:chns4} satisfying
\begin{alignat*}{2}
\qquad \quad
\phi&\in H^{2}(0,T;H^1(\Omega))\cap H^1(0,T;H^3(\Omega)),\quad & 
\mu& \in H^2(0,T;H^1(\Omega))\cap  L^2(0,T;H^{3}(\Omega)),
\\
\uu& \in  H^2(0,T;H^1(\Omega))\cap  H^1(0,T;H^{3}(\Omega)), & 
p & \in L^2(0,T;H^2(\Omega))\cap H^2(0,T;L^2(\Omega)).
\end{alignat*}
\end{itemize}
Under these assumptions, the solution of the corresponding initial value problem is unique. Further note that the regularity requirements are weaker than the ones used for the convergence analysis in \cite{Diegel2017}. Note that these regularity assumptions are already required to ensure optimal convergence rates for the interpolation and projection errors.
We can now state and prove the main results of this paper. 
%

%
\begin{theorem}
\label{thm:fulldisk}
Let (A0)--(A7) hold and $0<h,\tau \le c$ sufficiently small.
Then 
\begin{align}\label{eq:con_res_1}
\|\phi - \phi_{h,\tau}\|_{L^\infty(H^1)}
&+ \| \uu - \uu_{h,\tau}\|_{L^\infty(L^2)} \\
&+ \|\bar\mu - \bmu_{h,\tau}\|_{L^2(H^1)} + \|\bar\uu - \bar\uu_{h,\tau}\|_{L^2(H^1)}
\leq C \, (h^2 + \tau^2),\notag
\end{align}
with a constant $C$ depending only the bounds for the parameters and the solution, but not on $h$ and $\tau$.
For the choice $\tau=c' h$, the discrete solution is unique and 
further
\begin{align}
 \|\bar p - \bar p_{h,\tau} \|_{L^2(L^2)}  \leq C (h^2 + \tau^2).  \label{eq:con_res_2}
\end{align}
Recall that $\bar \mu = \bar \pi_\tau^0 \mu$ and $\bar p = \bar \pi_\tau^0 p$ denote the piecewise constant averages in time.
\end{theorem}
\noindent 
The main steps of the proof of these assertions are presented in Sections~\ref{sec:proof:thm:fulldisk_part1} and \ref{sec:proof:thm:fulldisk_part2}. 

%

\begin{remark}
Let us note that quantitative error estimates for fully discrete approximations to the Cahn-Hilliard-Navier-Stokes system were also established in \cite{CaiShen18,LiShen2020,LiuRiviere2020}.
In \cite{Diegel2017}, which is  closest to our results, error estimates $O(\tau^2 + h^q)$ are proven for a two-step implicit time stepping scheme combined with a finite element approximation in space using $P_{q}$--$P_{q}$--$P_{q+1}$--$P_{q}$ elements.
For $q=2$, the convergence rates correspond to that of Theorem~\ref{thm:fulldisk}, but the approximation orders for $\uu$ and $p$ are chosen higher by one order, in comparison to our method. 
Moreover, stronger smoothness assumptions on the solution are required in \cite{Diegel2017}.
None of the above works considers concentration dependent mobilities and viscosities, and convergence rates for the pressure cannot be found in the literature so far. 
\end{remark}

\section{Proof of Theorem~\ref{thm:well-posed}}
\label{sec:proof:thm:well-posed}

We start with verifying that any solution of Problem~\ref{prob:full} satisfies the mass- and energy balance, and then briefly comment on the existence of solutions. If not stated otherwise, we use $H^k=H^k(\Omega)$ and $L^p(X)=L^p(0,T;X)$ for abbreviation. For convenience of the reader, the proofs of some technical results are shifted to the appendix.

\subsection{Verification of \eqref{eq:mass_disc}--\eqref{eq:energy_disc}}
Let us first note that 
\begin{align*}
\int_\Omega \phi_{h,\tau}(t^n) \, dx
&= \int_\Omega \phi_{h,\tau}(t^{n-1}) \, dx 
 + \int_{t^{n-1}}^{t^n} \frac{d}{dt}\la  \phi_{h,\tau},1 \ra \, dt 
=  \int_\Omega \phi_{h,\tau}(t^{n-1}) \, dx + \la \dt \phi_{h,\tau}, 1\ra^n.
\end{align*}
The last term amounts to the left hand side of \eqref{eq:pg1} tested with $\bar \psi_{h,\tau} = 1$ on $(t^{n-1},t^n)$, which is an admissible test function. 
Since $\bar \psi_{h,\tau}$ is constant in space, the right hand side of \eqref{eq:pg1}, and hence also the last term in the above equation, vanishes. This yields \eqref{eq:mass_disc} for $m=n-1$; the case $m < n-1$ then follows readily by induction.  

In order to show the discrete energy--dissipation identity~\eqref{eq:energy_disc}, we note that
\begin{align*}
\E(\phi_{h,\tau}(t^n),\uu_{h,\tau}(t^n))
&= \E(\phi_{h,\tau}(t^{n-1}),\uu_{h,\tau}(t^{n-1})) + \int_{t^{n-1}}^{t^n} \frac{d}{dt} \E(\phi_{h,\tau},\uu_{h,\tau}) \,ds.
\end{align*}
From the definition of the energy functional, we further conclude that 
\begin{align*}
\frac{d}{dt} \E(\phi_{h,\tau},\uu_{h,\tau}) 
&= \int_\Omega \gamma \nabla \dt\phi_{h,\tau}\cdot \nabla  \phi_{h,\tau} + \dt \phi_{h,\tau} f'(\phi_{h,\tau})  + \dt \uu_{h,\tau} \cdot \uu_{h,\tau} \, dx.
\end{align*}
For ease of notation, the time dependence of the functions was not explicitly stated here. 
After integration in time and using the definition~\eqref{eq:abn} of $\la a, b\ra^n$, we then obtain  
\begin{align*}
\int_{t^{n-1}}^{t^n} \frac{d}{dt} &\E(\phi_{h,\tau},\uu_{h,\tau}) \,ds\\
&= \la \gamma \nabla \dt\phi_{h,\tau}, \nabla  \phi_{h,\tau} \ra^n + 
\la \dt \phi_{h,\tau}, f'(\phi_{h,\tau}) \ra^n + \la \dt \uu_{h,\tau}, \bar \uu_{h,\tau} \ra^n. 
\end{align*}
For the last term, we used that $\la \dt \uu_{h,\tau},\uu_{h,\tau}\ra^n = \la \dt \uu_{h,\tau},\bar \uu_{h,\tau}\ra^n$ with $\bar \uu_{h,\tau}=\bar \pi_\tau^0 \uu_{h,\tau}$, since the function $\dt \uu_{h,\tau}$ is piecewise constant in time.
The three terms on the right hand side of the above formula also appear in \eqref{eq:pg2}--\eqref{eq:pg3}, when testing with $\bar \xi_{h,\tau} = \dt \phi_{h,\tau}$ and $\bar \vv_{h,\tau} = \bar \uu_{h,\tau}$, which is admissible.
Testing the identities \eqref{eq:pg1} and \eqref{eq:pg4} with $\bar \psi_{h,\tau}=\bar \mu_{h,\tau}$ and $\bar q_{h,\tau} = \bar p_{h,\tau}$, which is again admissible, and adding the resulting equations then leads to 
\begin{align*}
\int_{t^{n-1}}^{t^n} \frac{d}{dt} \E(\phi_{h,\tau},\uu_{h,\tau}) \,ds
&= - \la b(\bar \phi_{h,\tau}) \nabla \bar \mu_{h,\tau}, \nabla \bar \mu_{h,\tau}\ra^n - \la \eta(\bar \phi_{h,\tau}) \nabla \bar \uu_{h,\tau}, \nabla \bar \uu_{h,\tau} \ra^n.
\end{align*}
Here we used that $\la \mathbf{a} \cdot \nabla \mathbf{b}, \mathbf{c} \ra_{skw}=0$ for $\mathbf{b}=\mathbf{c}$, and that all the remaining terms cancel out each other. 
Together with the definition of $\D_\phi(\mu,\uu)$, this leads to \eqref{eq:energy_disc} for $m=n-1$; the case $m < n-1$ again follows by induction.
\qed

\subsection{A-priori bounds}

From the basic properties \eqref{eq:mass_disc}--\eqref{eq:energy_disc} and our assumptions on the model parameters, we can immediately deduce the following result.
\begin{lemma} \label{lem:apriori}
Let (A0)--(A7) hold. Then any solution of Problem~\ref{prob:full} satisfies
\begin{align} \label{eq:bounds}
\|\phi_{h,\tau}\|_{L^\infty(H^1)} + \|\uu_{h,\tau}\|_{L^\infty(L^2)} + \|\bar \uu_{h,\tau}\|_{L^2(H^1)} + \|\bar \mu_{h,\tau}\|_{L^2(H^1)} \le \hat C
\end{align}
with a constant $\hat C$ depending only on the bounds for the parameters and the initial data.
Furthermore, the pressure is bounded by $\|\bar p_{h,\tau}\|_{L^2(L^2)} \le C(\hat C,\tau)$.
Recall that $\bar \uu_{h,\tau}$, $\bar \mu_{h,\tau}$, and $\bar p_{h,\tau}$ are piecewise constant functions in time.
\end{lemma}
\begin{proof}
From the choice of the discrete initial values and the properties of the projection operators stated in Appendix~\ref{app:proj}, we see that 
\begin{align*}
\E(\phi_{h,\tau}(0),\uu_{h,\tau}(0)) \le c
\end{align*}
with constant $c$ independent of $h$ and $\tau$. 
From \eqref{eq:energy_disc} and assumptions~(A2)--(A4), and noting that $\phi_{h,\tau}$ and $\uu_{h,\tau}$ are piecewise linear in time, we then deduce that $\|\nabla \phi_{h,\tau}\|_{L^\infty(L^2)}$, $\|\uu_{h,\tau}\|_{L^\infty(L^2)}$ and $\|\nabla \bmu_{h,\tau}\|_{L^2(L^2)}$, $\|\nabla \bar \uu_{h,\tau}\|_{L^2(L^2)}$ are uniformly bounded, independent of $h$ and $\tau$. 
This already yields the bounds for the velocity.
From \eqref{eq:mass_disc}, the estimate for the gradient, and the Poincar\'e inequality, we further deduce that also $\|\phi_{h,\tau}\|_{L^\infty(H^1)}$ is uniformly bounded.
By testing \eqref{eq:pg2} with $\bar \xi_{h,\tau}=1|_{[t^{n-1},t^n]}$, we then see that 
\begin{align*}
\tau \int_{\Omega} \bmu_{h}(t^{n-1/2}) \, dx 
&= \la \bmu_{h,\tau},1\ra^n 
 = \gamma \la \nabla \bar \phi_{h,\tau}, \nabla 1 \ra^n + \la f(\bar \phi_{h,\tau}),1\ra^n \\
&\le \la f(\bar \phi_{h,\tau}) - f(0), 1\ra^n
 \le C' \tau \| \phi_{h,\tau}\|_{L^\infty(H^1)}.
\end{align*}
Here $t^{n-1/2}-t^n=\tau/2$, and in the last step, we used the mean-value theorem, H\"older's inequality, and assumption~(A3).
Together with the bound for $\nabla \bmu_{h,\tau}$ and the Poincar\'e inequality, we obtain the remaining estimate for chemical potential. 
Further note that
\begin{align} \label{eq:infsup}
    \sup_{\vv_h \in \Vh^d} \frac{\la \div \vv_h,q_h\ra}{\|\vv_h\|_{H^1}} \ge \beta \|q_h\|_{L^2} \qquad \forall q_h \in \Qh,
\end{align}
since we use a Taylor-Hood finite element pair for velocity and pressure; see e.g. \cite{GiraultRaviart,John2016}. 
Identity~\eqref{eq:pg3} then allows to derive a corresponding bound for the discrete pressure.
\end{proof}

\subsection{Existence of solutions}

We only sketch the main argument here:
In the $n$th time step, we assume $\phi_{h,\tau}(t^{n-1})$ and $\uu_{h,\tau}(t^{n-1})$ to be given, and seek for $\dt \phi_h^n := \dt \phi_{h,\tau}(t^{n-})$, $\uu_h^n:=\uu_{h,\tau}(t^n)$, $\mu_h^n:=\bar \mu_{h,\tau}(t^{n-})$, $p_h^n:=\bar p_{h,\tau}(t^{n-})$ in the respective finite element spaces.
As usual, $f(t^-)$ denotes the limit from the left.
By some tedious but elementary manipulations, the system \eqref{eq:pg1}--\eqref{eq:pg4} can be brought into an equivalent fixed-point form 
\begin{align}
\la x_h,y_h\ra = \la T_h x_h, y_h\ra \qquad \forall y_h \in Y_h
\end{align}
with appropriate finite dimensional operator $T_h : Y_h \to Y_h$ and $x_h=(\dt \phi_h^n,\mu_h^n,\uu_h^n,p_h^n)$. 
This variational form is equivalent to $x_h = T_h x_h$. 
By testing with $y_h=x_h$ and using the same arguments as employed in the derivation of the a-priori estimates above, one can verify that any solution of $x_h = \lambda T_h x_h$ with $0 \le \lambda \le 1$ satisfies the same uniform a-priori estimates. 
Existence of a solution $x_h$ for $\lambda=1$ then follows from the Leray-Schauder principle; see \cite[Theorem~6.A]{Zeidler_IIB}.
This shows the claim for a single time step~$n$, and we can again use induction to cover all time steps. Note that no restrictions on the choice of $h,\tau >0$ are necessary for the existence, since a-priori bounds are obtained from the energy stability, which holds uniformly in the discretization parameters. \qed

\section{Proof of Theorem~\ref{thm:fulldisk}: Part~1}
\label{sec:proof:thm:fulldisk_part1}

The aim of this section is to establish the error bounds \eqref{eq:con_res_1}.
This is accomplished by splitting into projection errors and discrete errors, and then estimating both contributions by appropriate interpolation error bounds and a nonlinear discrete stability estimate. 
%
The verification of some technical details is provided in the appendix.  
%
%

\subsection{Projection errors}

For splitting the discretization error into projection and discrete components, we utilize appropriately defined discrete approximations for a sufficiently smooth periodic solution of \eqref{eq:chns1}--\eqref{eq:chns4}, obtained by
\begin{align} \label{eq:fullproj}
\hat \phi_{h,\tau} = \I_\tau^1 \pi_h^1 \phi, \qquad 
\hbmu_{h,\tau} = \bar \pi_\tau^0 \pi_h^0 \mu, \qquad 
\hat\uu_{h,\tau}=\I_\tau^1\mathbf{P}^1_h \uu, \quad  \text{ and} \quad \hbp_{h,\tau} = \bar \pi_\tau^0 \pi^0_{p,h} p.
\end{align}
Let us recall that $\I_\tau^1$ and  $\bar\pi_\tau^0$ denote the piecewise linear interpolation and the piecewise constant projection in time, while $\pi_h^0$, $\pi_h^1$, and $\mathbf{P}^1_h$ are the $L^2$-, $H^1$- and Stokes-projection in space, respectively; see Section~\ref{sec:prelim}.
As a direct consequence of well-known properties of these operators, summarized in Appendix~\ref{app:proj}, one can deduce the following bounds.
\begin{lemma}\label{lem:projerr}
Let (A5)--(A7) hold. Then
\begin{align*}
&\|\hat\phi_{h,\tau} - \phi\|^2_{L^\infty(H^1)} \leq C(\tau^4 + h^4),& 
&\|\hbmu - \bar\mu\|^2_{L^2(H^1)} \leq Ch^4\\
&\|\hat\uu_{h,\tau} - \uu\|^2_{L^\infty(L^2)} \leq C(\tau^4 + h^4),& 
&\|\hbu_{h,\tau}-\bar\uu\|^2_{L^2(H^1)} + \|\hbp_{h,\tau} - \bar p\|^2_{L^2(L^2)} \leq Ch^4, \\
&\|\dt \hat\phi_{h,\tau}-\bar \pi_\tau^0 (\dt \phi)\|^2_{L^2(H^{-1})} \leq C h^4,&
&\|\dt\hat\uu_{h,\tau}-\bar \pi_\tau^0 (\dt \uu)\|^2_{L^2(H^{-1})} \leq Ch^4.
\end{align*}
The constant $C$ in these estimates only depends on the bounds in the assumptions.
\end{lemma}

In order to establish the error bounds \eqref{eq:con_res_1} of Theorem~\ref{thm:fulldisk}, it remains to apply the triangle inequality and to derive corresponding estimates for the error between the discrete solution and the projections defined in \eqref{eq:fullproj}. 
This is achieved by interpreting the latter as a solution of a perturbed discrete problem, and then using stability of the discrete problem.

\subsection{Error equation and discrete residuals}

As a first step, we show that the above projections can also be interpreted as the solutions of a perturbed discrete problem.
\begin{lemma} \label{lem:error_residual}
With the choice \eqref{eq:fullproj} of discrete approximations, we have 
\begin{align}
\la \dt\hat\phi_{h,\tau}, \bar\psi_{h,\tau}\ra^n
&= \la \phi_{h,\tau}\hbu_{h,\tau},\nabla\bar\psi_{h,\tau}\ra^n - \la b(\bar\phi_{h,\tau})\nabla\hbmu_{h,\tau},\nabla\bar\psi_{h,\tau} \ra^n + 
\la \bar r_{1,h,\tau}, \bar\psi_{h,\tau}\ra^n
,  \label{eq:disc_pert1}
\\
\la \hbmu_{h,\tau},\bar\xi_{h,\tau} \ra^n &= \gamma\la \nabla\hat\phi_{h,\tau},\nabla\bar\xi_{h,\tau} \ra^n  + \la f'(\hat\phi_{h,\tau}),\bar\xi_{h,\tau}\ra^n 
+ \la \bar r_{2,h,\tau}, \bar\xi_{h,\tau}\ra^n, \label{eq:disc_pert2}
\\
\la \dt \hat\uu_{h,\tau} , \bar\vv_h \ra^n &= \la \uu_{h,\tau} \cdot \nabla \hbu_{h,\tau},\bar\vv_{h,\tau}\ra_\skw^n -\la \eta(\bar\phi_{h,\tau}) \nabla\hbu_{h,\tau}, \nabla\bar\vv_{h,\tau} \ra^n,
 \label{eq:disc_pert3} \\
&\qquad \qquad 
-\langle \hbp_{h,\tau}, \div \bar\vv_{h,\tau}\ra^n
+ \la \phi_{h,\tau} \nabla\hbmu_{h,\tau},\bar\vv_{h,\tau}\ra^n+ \la \rr_{3,h,\tau},\bar \vv_{h,\tau}\ra^n, \notag 
\\
0 &= \la \div \hat\uu_{h,\tau}, \bar q_{h,\tau} \ra^n , \label{eq:disc_pert4}
\end{align}
for all test functions $\bar \psi_{h,\tau}, \bar \xi_{h,\tau} \in P_0(\Itau;\Vh$, $\bar\vv_{h,\tau} \in P_0(\Itau;\Vh^d)$, and $\bar q_{h,\tau} \in P_0(\Itau;\Qh)$, 
with discrete residuals $\bar r_{1,h,\tau}, \bar r_{2,h,\tau} \in P_0(\Itau;\Vh)$ and $\bar \rr_{3,h,\tau} \in P_0(\Itau;\Vh^d)$ defined by
\begin{align*}
\la \bar r_{1,h,\tau}, \bar \psi_{h,\tau}\ra^n
&= \la \dt (\pi_h^1 \phi - \phi), \bar \psi_{h,\tau}\ra^n + \la b(\bar\phi_{h,\tau}) \nabla \hbmu_{h,\tau} - b(\phi) \nabla \mu, \nabla \bar \psi_{h,\tau}\ra^n  
\\
& \qquad \qquad - \la (\phi_{h,\tau}\hbu_{h,\tau} - \phi\uu,\nabla\bar\psi_{h,\tau} \ra^n,
\\
\la  \bar r_{2,h,\tau}, \bar \xi_{h,\tau} \ra^n
&= \la \hbmu_{h,\tau} - I_\tau^1 \mu, \bar \xi_{h,\tau}\ra^n 
      + \gamma \la \nabla (\hat \phi_{h,\tau} - \I_\tau^1 \phi), \nabla \bar \xi_{h,\tau}\ra^n 
\\
& \qquad \qquad  + \la f'(\hat \phi_{h,\tau}) - \I_\tau^1 f'(\phi), \bar \xi_{h,\tau}\ra^n,
\\
\la \bar \rr_{3,h,\tau},\bar \vv_{h,\tau} \ra^n 
&= \la \dt\mathbf{P}^1_h\uu-\dt\uu,\bar\vv_{h,\tau} \ra^n + \la \eta(\bar\phi_{h,\tau})\nabla\hbu_{h,\tau} - \eta(\phi)\nabla\uu,\nabla\bar\vv_{h,\tau} \ra^n 
\\
& \qquad \qquad - \la \hbp_{h,\tau}-p,\div(\bar\vv_{h,\tau}) \ra^n + 
\la \uu_{h,\tau} \cdot \nabla \hbu_{h,\tau} - \uu \cdot \nabla \uu,\bar\vv_{h,\tau}\ra_\skw^n
\\
& \qquad \qquad 
+ \la \phi_{h,\tau} \nabla\hbmu_{h,\tau} - \phi \nabla \mu, \bar\vv_{h,\tau}\ra^n.
\end{align*}
\end{lemma}
\begin{proof}
The result follows immediately by plugging the projections into the discrete variational problem \eqref{eq:pg1}--\eqref{eq:pg4} and using some elementary properties of the projections as well as the variational characterization \eqref{eq:weak1}--\eqref{eq:weak4} of the true solution. 
\end{proof}

\subsection{Discrete stability and residual estimates}

In order to measure the discrete error, i.e., the difference between the discrete solution and the projection of the continuous solution, we use a regularized relative energy functional, which is defined by 
\begin{align} \label{eq:Ealpha}
\E_\alpha(\phi,\uu|\hat \phi,\hat \uu)
&:= \int_\Omega \frac{\gamma}{2} |\nabla \phi - \nabla \hat \phi|^2 + f(\phi|\hat \phi) + \frac{\alpha}{2} |\phi - \hat \phi|^2 + \frac{1}{2} |\uu - \hat \uu|^2 \, dx,
\end{align}
where $f(\phi|\hat \phi) := f(\phi) - f(\hat \phi) - f'(\hat \phi) (\phi - \hat \phi)$ is the second order Taylor remainder of $f$. The parameter $\alpha$ is chosen by $\alpha = \max\{\gamma,\gamma+f_1\}$, 
such that $f(\phi|\hat \phi) + \frac{\alpha}{2}|\phi - \hat \phi|^2$ becomes strongly convex. 
We will make frequent use of the following auxiliary result. 
\begin{lemma} \label{lem:equiv}
Let (A0)--(A4) hold. Then
\begin{align} \label{eq:lower_bound_rel}
c_0 \, (\|\phi - \hat \phi\|_{H^1}^2 + \|\uu - \hat \uu\|_{L^2}^2)  \le \E_\alpha(\phi,\uu|\hat \phi,\uu) \le C_0' \, (\|\phi - \hat \phi\|_{H^1}^2 + \|\uu - \hat \uu\|_{L^2}^2)
\end{align}
for all $\phi,\hat \phi \in H^1(\Omega)$ and all $\uu, \hat \uu \in L^2(\Omega)$ with $C_0'=C_0 (1+\|\phi\|_{H^1}^2+\|\hat \phi\|_{H^1}^2)$ and generic constants constants $c_0,C_0$ depending only on the bounds in the assumptions.
\end{lemma}
\begin{proof}
From assumption (A3) and the choice of $\alpha$, we can see that 
\begin{align*}
    -\frac{\alpha}{4} |\phi - \hat \phi|^2 \le f(\phi|\hat \phi) \le C_f (1 + |\phi|^2 + |\hat \phi|^2)|\phi - \hat \phi|^2,  
\end{align*}
with constant $C_f$ independent of $\phi$, $\hat \phi$. The assertion of the lemma then follows by integration over $\Omega$ and elementary arguments. 
\end{proof}

The following result is the key ingredient for the estimation of the discrete errors.

\begin{lemma}[Relative energy estimate] \label{lem:fullstab} 
%
Let (A0)--(A6) hold.
Then 
\begin{align*}
 \E_\alpha(\phi_{h,\tau},&\uu_{h,\tau}|\hat \phi_{h,\tau},\hat\uu_{h,\tau}) \, \Big|_{t^{n-1}}^{t^n} + \frac{3}{4} \int_{t^{n-1}}^{t^n} \D_{\bar \phi_{h,\tau}}(\bmu_{h,\tau}-\hbmu_{h,\tau},  \bu_{h,\tau}-\hbu_{h,\tau}) \, ds 
 \\
 & \leq \bar c \int_{t^{n-1}}^{t^n} \E_\alpha(\phi_{h,\tau},\uu_{h,\tau}|\hat \phi_{h,\tau},\hat\uu_{h,\tau})\, ds \\ 
 & \qquad \qquad \qquad + \bar C \int_{t^{n-1}}^{t^{n}} \|\bar r_{1,h,\tau}\|_{H^{-1}}^2 + \|\bar r_{2,h,\tau}\|_{H^{1}}^2 + \|\bar \rr_{3,h,\tau}\|_{\Zh^*}^2 \, ds,
\end{align*}
with constants $\bar c$, $\bar C$ independent of $h$ and $\tau$, 
and the dual norm defined by
\begin{align*}
    \|\rr\|_{\Zh^*}:=\sup_{\vv_h \in \Zh} \frac{\la \rr, \vv_h\ra}{\|\vv_h\|_{H^1}},
\end{align*}
with $\Zh=\{\vv_h \in \Vh^d : \la \div \vv_h, q_h\ra=0 \ \forall q_h \in \Qh\}$ the discrete divergence free velocities.
\end{lemma}

A detailed proof of this estimate will be given in Appendix~\ref{app:fullstab}. We only note at this point that the use of the relative energy and dissipation functionals is tailored to the nonlinearity of the problem and greatly simplifies the analysis. 
The second main ingredient for the estimation of the discrete error are the following bounds for the residuals.
\begin{lemma}\label{lem:residual}
Let (A0)--(A7) hold. Then 
\begin{align*}
&\int_{t^{n-1}}^{t^n} \|\bar r_{1,h,\tau}\|^2_{H^{-1}}
+ \|\bar r_{2,h,\tau}\|^2_{H^{1}} + \|\bar \rr_{3,h,\tau}\|^2_{\Zh^*} \, ds \\ 
& \qquad \qquad \le 
\hat c \int_{t^{n-1}}^{t^n} \E_\alpha(\phi_{h,\tau},\uu_{h,\tau}|\hat \phi_{h,\tau},\hat\uu_{h,\tau}) \, ds \\
& \qquad \qquad \qquad \qquad + \frac{1}{4 \bar C} \int_{t^{n-1}}^{t^n} \D_{\bar \phi_{h,\tau}}(\bmu_{h,\tau}-\hbmu_{h,\tau},  \bu_{h,\tau}-\hbu_{h,\tau}) \, ds + \hat C (h^4 + \tau^4),
\end{align*}
with constant $\bar C$ from Lemma~\ref{lem:fullstab}, and constants $\hat c$, $\hat C$ again independent of $h$ and $\tau$. 
\end{lemma}
The proof of this assertion is rather technical but essentially follows from interpolation and projection error estimates, and will be presented in detail in Appendix~\ref{app:residual}.

We will further make use of the following basic estimate, which follows readily from convexity of the energy functional and piecewise linearity of the discrete functions in time.
\begin{lemma}\label{lem:equivnorms}
    Let (A1)--(A6) hold. Then
    \begin{align} \label{eq:equivnorms}
      \int_{t^{n-1}}^{t^n} &\E_\alpha(\phi_{h,\tau},\uu_{h,\tau}|\hat \phi_{h,\tau},\hat\uu_{h,\tau})\, ds  \\
      &\leq \tau\left(\E_\alpha(\phi_{h,\tau},\uu_{h,\tau}|\hat \phi_{h,\tau},\hat\uu_{h,\tau})(t^{n-1}) + \E_\alpha(\phi_{h,\tau},\uu_{h,\tau}|\hat \phi_{h,\tau},\hat\uu_{h,\tau})(t^n) \right). \notag
    \end{align}
\end{lemma}

\subsection{Discrete error bounds and proof of Theorem~\ref{thm:fulldisk}}

By combination of the auxiliary results above, we can now prove the following bound for the discrete error. 

\begin{lemma}
\label{lem:diskerr}
Let (A0)--(A7) hold. Then 
\begin{align*}
\|\phi_{h,\tau} - \hat\phi_{h,\tau}&\|_{L^\infty(H^1)} +  \|\uu_{h,\tau} - \hat\uu_{h,\tau}\|_{L^\infty(L^2)} \\
&+ \|\bmu_{h,\tau} - \hbmu_{h,\tau}\|_{L^2(H^1)} + \|\bu_{h,\tau} - \hbu_{h,\tau}\|_{L^2(H^1)}
\leq C'_T(h^2 + \tau^2),
\end{align*}
with constant $C'_T$ independent of $h$ and $\tau$ but depending on the final time, the model parameters and  bounds for the solution provided by assumption~(A7).
\end{lemma}
\begin{proof}
With Lemma \ref{lem:equivnorms} and using $2\tau \le 1/(\bar c + \hat c)$, we get $(1- \bar c \tau) \ge e^{-\lambda \tau}$ and $(1+\bar c \tau) \le e^{\lambda \tau}$ with some constant $\lambda$ depending only on $\bar c+\hat c$. 
The previous results then lead to 
\begin{align*}
e^{-\lambda \tau} u^n &+ b^n \le e^{\lambda \tau} u^{n-1} + d^n
\end{align*}
with
$u^n =\E_\alpha(\phi_{h,\tau},\uu_{h,\tau}|\hat \phi_{h,\tau},\hat\uu_{h,\tau})(t^n)$,
$b^n =  \frac{1}{2} \int_{t^{n-1}}^{t^n} \D_{\bar \phi_{h,\tau}}(\bmu_{h,\tau}-\hbmu_{h,\tau}, \bu_{h,\tau}-\hbu_{h,\tau}) \, ds$, 
and 
$d^n =  \bar C \hat C (h^2 + \tau^2)^2$.
By the discrete Gronwall inequality \eqref{eq:discgronwall}, we thus obtain 
\begin{align*}
\E_\alpha(\phi_{h,\tau},\uu_{h,\tau}|,\hat \phi_{h,\tau},\hat \uu_{h,\tau})(t^n) &+ \int_0^{t^n} \D_{\bar \phi_{h,\tau}}(\bmu_{h,\tau} - \hbmu_{h,\tau}, \uu_{h,\tau} - \hat \uu_{h,\tau}) \, ds \\
&\le C_T \E_\alpha(\phi_{h,\tau},\uu_{h,\tau}|,\hat \phi_{h,\tau},\hat \uu_{h,\tau})(0) + C_T' (h^4 + \tau^4).
\end{align*}
Due to the choice of the initial values in Problem~\ref{prob:full}, the first term on the right hand side vanishes.
Lemma~\ref{lem:equiv} then allows us to estimate the relative energy from below by the norm. 
With the lower bounds in (A1) and (A2), we can further estimate the dissipation term from below to obtain the bound 
\begin{align*}
\|\nabla (\bmu_{h,\tau} - \hbmu_{h,\tau})\|_{L^2(L^2)}^2 + \|\nabla (\uu_{h,\tau} - \hbu_{h,\tau})\|_{L^2(L^2)}^2 \le C_T''(h^4 + \tau^4).
\end{align*}
In Appendix~\ref{subsec:mu}, we show that a corresponding estimate also holds for $\|\bmu_{h,\tau} - \hbmu_{h,\tau}\|_{L^2(L^2)}^2$, which completes the proof of the lemma. 
\end{proof}

\subsection*{Proof of Theorem~\ref{thm:fulldisk}, Part~1}
The error estimates \eqref{eq:con_res_1} now follow readily by the triangle inequality, and combination of Lemma~\ref{lem:projerr} and \ref{lem:diskerr}. 
\qed

\section{Proof of Theorem~\ref{thm:fulldisk}: Part~2}
\label{sec:proof:thm:fulldisk_part2}

With the help of the discrete stability estimate, see Lemma~\ref{lem:fullstab}, we now show that, under a mild restriction $\tau = c' h$ on the time step, the solution of Problem~\ref{prob:full} is unique.
Moreover, we establish optimal convergence rates also for the error in the pressure.

\subsection{Uniqueness of discrete solutions}

As before, we denote by $(\hat\phi_{h,\tau},\hbmu_{h,\tau},\hat\uu_{h,\tau},\hbp_{h,\tau})$ a given discrete solution of Problem~\ref{prob:full}, and we assume that (A0)--(A7) hold. 
Let us start our considerations with an elementary observation.
\begin{lemma}
Let $(\hat\phi_{h,\tau},\hbmu_{h,\tau},\hat\uu_{h,\tau},\hbp_{h,\tau})$ be another solution of Problem~\ref{prob:full}. 
Then the perturbed equations \eqref{eq:disc_pert1}--\eqref{eq:disc_pert4} hold with residuals 
$\bar r_{2,h,\tau}=0$ and 
\begin{align*}
\la \bar r_{1,h,\tau},\bar\psi_{h,\tau}\ra^n 
&= \la (\phi_{h,\tau}-\hat\phi_{h,\tau})\hbu_{h,\tau},\nabla\bar\psi_{h,\tau}\ra^n 
+ \la (b(\bar\phi_{h,\tau})-b(\hbphi_{h,\tau}))\nabla\hbmu_{h,\tau}, \nabla \bar\psi_{h,\tau}\ra^n,
\\
\la \bar \rr_{3,h,\tau},\bar\vv_{h,\tau} \ra^n 
&=  \la (\uu_{h,\tau}-\hat\uu_{h,\tau}) \cdot \nabla \hbu_{h,\tau},\bar\vv_{h,\tau}\ra^n 
+ \la (\eta(\bar\phi_{h,\tau}) - \eta(\hbphi_{h,\tau})) \nabla\hbu_{h,\tau},\nabla\bar\vv_{h,\tau} \ra^n \\
& \qquad \qquad - \la(\phi_{h,\tau}-\hat\phi_{h,\tau})\bar\vv_{h,\tau},\nabla\hbmu_{h,\tau}\ra^n.
\end{align*}
\end{lemma}
The claim follows immediately from the variational identities defining the discrete solutions.
%
%
Similar to the previous section, we next state bounds for the corresponding residuals. 
\begin{lemma} \label{lem:res2}
Under the above assumptions, one has 
\begin{align*}
\int_{t^{n-1}}^{t^{n}}\|\bar r_{1,h,\tau}\|_{H^{-1}}^2 &+ \|\bar \rr_{3,h,\tau}\|^2_{H^{1}} \, ds 
\le \frac{1}{4 \bar C} \int_{t^{n-1}}^{t^n} \D_{\bar \phi_{h,\tau}}(\bmu_{h,\tau}-\hbmu_{h,\tau},  \bu_{h,\tau}-\hbu_{h,\tau}) \, ds \\
&+ \tilde c \int_{t^{n-1}}^{t^n} \E_\alpha(\phi_{h,\tau},\uu_{h,\tau}|\hat \phi_{h,\tau},\hat\uu_{h,\tau}) \, ds,
\end{align*}
with $\bar C$ denoting the constant of Lemma~\ref{lem:fullstab}.
Moreover, for any $0<h,\tau \le \tau_0$ sufficiently small, the constant $\tilde c$ can be chosen independent of $h$ and $\tau$. 
\end{lemma}
The proof of this assertion is provided in Appendix~\ref{app:residual}. 
By invoking the discrete stability estimate of Lemma~\ref{lem:fullstab}, which holds also in the present situation,
further noting that $\phi_{0,h}=\hat\phi_{0,h}$, $\uu_{0,h}=\hat\uu_{0,h}$, and proceeding similar to the proof of Lemma~\ref{lem:diskerr}, we see that  
\begin{align*}
    \mathcal{E}_\alpha(\phi_{h,\tau},\uu_{h,\tau}|\hat\phi_{h,\tau},\hat\uu_{h,\tau})(t^n) + \int_0^{t^n} \mathcal{D}_{\phi_{h,\tau}}(\bmu_{h,\tau}-\hbmu_{h,\tau},\bu_{h,\tau}-\hbu_{h,\tau}) \, ds \leq 0.
\end{align*}
From the lower bounds for the relative energy and dissipation functionals, we therefore conclude that  $\phi_{h,\tau} \equiv \hat\phi_{h,\tau}, \uu_{h,\tau} \equiv \hat\uu_{h,\tau}$, and in consequence also $\bmu_{h,\tau}\equiv\hbmu_{h,\tau}$. 
The uniqueness of the pressure finally follows with the arguments used in the proof of Theorem~\ref{thm:well-posed}. 
\qed

\subsection{Estimates for pressure error}\label{sec:ch7:sec51}


As a last step of our convergence analysis, we now establish bounds for the error in the pressure. 
We again split the error via
\begin{align*}
    \|\bar p - \bar p_{h,\tau}\|_{L^2(L^2)} 
    \le \|\bar p - \hbp_{h,\tau}\|_{L^2(L^2)} + \|\hbp_{h,\tau} - \bar p_{h,\tau}\|_{L^2(L^2)}, 
\end{align*}
into a projection error and a discrete error component. 
Recall that $\bar p = \bar \pi_\tau^0 p$ is the piecewise constant projection in time and $\hbp_{h,\tau} = \bar \pi_\tau^0 \pi_h^0 p$ involves another $L^2$-projection in space. 
By standard estimates for the projection operators, see Appendix~\ref{app:proj}, one obtains 
\begin{align*}
\|\bar p - \hbp_{h,\tau}\|_{L^2(L^2)} \le \|p - \pi_h^0 p\|_{L^2(L^2))} \le C h^2 \|p\|_2 \le C' h^2.    
\end{align*}
In order to bound the discrete error component, we use the following auxiliary result.
\begin{lemma}\label{lem:perropart1}
Let (A0)--(A7) hold and $\tau = c h$ be sufficiently small. Then 
\begin{align*}
\|\bar p_{h,\tau} - \hbp_{h,\tau}\|_{L^2(L^2)} 
\leq \|\dt\uu_{h,\tau} -\dt\hat\uu_{h,\tau}\|_{L^2(H^{-1})} + C_T'' (h^2 + \tau^2).
\end{align*}
\end{lemma}
\begin{proof}
From the discrete inf-sup condition \eqref{eq:infsup} and using the variational identities \eqref{eq:pg3} and \eqref{eq:disc_pert3}, we can deduce that
\begin{align}\label{eq:press_est_full_disc}
&\beta^2 \int_{t^{n-1}}^{t^n} \|\bar p_{h,\tau} - \hbp_{h,\tau} \|_{L^2}^2 \, ds \\
&\qquad \leq  C\int_{t^{n-1}}^{t^n}\norm{\dt\uu_{h,\tau} -\dt\hat\uu_{h,\tau}}_{\Zh^*}^2 +\|\bar\uu_{h,\tau}\|_{L^\infty}^2\norm{\bu_{h,\tau}-\hbu_{h,\tau}}_{H^1}^2    \nonumber\\
&\qquad \qquad + \|\bu_{h,\tau}-\hbu_{h,\tau}\|_{H^1}^2 + \|\bar\phi_{h,\tau}\|^2_{L^3} \|\nabla(\bmu_{h,\tau}-\hbmu_{h,\tau})\|_{L^2}^2 + \|\bar \rr_{3,h,\tau}\|_{\Zh^*}^2 \, ds \nonumber\\
& \qquad \leq C \int_0^t\norm{\dt\uu_h -\dt\hat\uu_h}_{\Zh^*}^2 +  \|\bar\uu_{h,\tau}\|_{L^\infty}^2 \|\bu_{h,\tau}-\hbu_{h,\tau}\|_{H^1}^2 \, ds +  C'_T(h^4 + \tau^4). \nonumber
\end{align}
In the final step, we used the convergence estimates  \eqref{eq:con_res_1} of Theorem~\ref{thm:fulldisk} for the last term.
The second term can further be bounded by
\begin{align*}
\int_{t^{n-1}}^{t^n} \| \bar\uu_{h,\tau}\|_{L^\infty}^2 & \|\bu_{h,\tau} -\hbu_{h,\tau}\|_{H^1}^2\, ds \\
&\leq \int_{t^{n-1}}^{t^n} (\|\bar\uu_{h,\tau}-\hbu_{h,\tau}\|_{L^\infty}^2  + \|\hbu_{h,\tau}\|_{L^\infty}^2)\norm{\bu_{h,\tau}-\hbu_{h,\tau}}_{H^1}^2 \, ds \\
&\leq (\norm{\bar\uu_{h,\tau}-\hbu_{h,\tau}}_{L^\infty(L^\infty)}^2  + \norm{\hbu_{h,\tau}}_{L^\infty(L^\infty)}^2) \int_{t^{n-1}}^{t^n}\norm{\bu_{h,\tau}-\hbu_{h,\tau}}_{H^1}^2 \, ds \\
&\leq (\norm{\bar\uu_{h,\tau}-\hbu_{h,\tau}}_{L^\infty(L^\infty)}^2  + \norm{\hbu_{h,\tau}}_{L^\infty(L^\infty)}^2)C'_T(h^4 + \tau^4)
= (*).
\end{align*}
For the first term we use the inverse inequality \eqref{eq:inverse} in space with $p=\infty, q=2, d\leq 3$. The second term is uniformly bounded, which follows from the regularity of the solution and the properties of the projection operators. 
Together with \eqref{eq:con_res_1}, we then obtain
\begin{align*}
(*) 
&\leq C(h^{-3}\norm{\bar\uu_{h,\tau}-\hbu_{h,\tau}}_{L^\infty(L^2)}^2  + 1)(h^4 + \tau^4)  \\
&\leq Ch^{-3}(h^4 + \tau^4)(h^4 + \tau^4) + C(h^4+ \tau^4)
\le C (h^4 + \tau^4),
\end{align*}
where we used $\tau \approx h$ in the last step. 
This already proves the assertion of the lemma.
\end{proof}

In order to complete the error estimate for the pressure, we use the following observation, which essentially follows the arguments given in \cite{Ayuso,John2016}. 
\begin{lemma} \label{lem:time_derivative}
Under the assumptions of the previous lemma, there holds
\begin{align*}
\|\dt\uu_{h,\tau} -\dt\hat\uu_{h,\tau}\|_{L^2(H^{-1})} 
\le C'_T \, (h^2 + \tau^2).
\end{align*}
\end{lemma}
A detailed proof of this assertion is provided in Appendix~\ref{app:time-derivative}. 

\subsection*{Completion of the proof}
By combination of the previous results, we now immediately obtain the estimate \eqref{eq:con_res_2} of Theorem~\ref{thm:fulldisk}, which concludes our theoretical considerations.
\qed

\section{Numerical tests}\label{sec:num}
For confirmation of our theoretical findings, we now present some computational results and convergence rates for a typical test problem. 

\subsection*{Model problem}

We consider the domain $\Omega=(0,1)^2$, which is identified with the two-torus $\mathbb{T}^2$, i.e., \eqref{eq:chns1}--\eqref{eq:chns4} is complemented by periodic boundary conditions. 
We further set $T=2$ for the final time and choose the  following initial conditions
\begin{equation*}
 \phi_0 = 0.5+0.25\cos(2\pi x)\cos(2\pi y), \uu_0=0.25(-\sin(\pi x)^2\sin(2\pi y),\sin(\pi y)^2\sin(2\pi x)).
\end{equation*}
The remaining model parameters are chosen as $\gamma=0.001$, $f(\phi)=(\phi-0.99)^2(\phi-0.01)^2$, $b(\phi)= 0.1(1-\phi)^2\phi^2 + 10^{-3}$, and $\eta = 2.5\cdot 10^{-4}(\phi+1)^2 + 10^{-3}$.
%


All numerical results presented in the following are based on the scheme defined in Problem~\ref{prob:full}. Since all nonlinearities appearing here are polynomial, all integrals arising in the implementation can be computed exactly. The nonlinear systems at every time-step are solved by the Newton method with tolerance $10^{-12}$. The linear systems in every Newton-step are solved by a sparse-direct solver. The implementation was entirely written in \textsc{MatLab}.

Some snapshots of the evolution of the phase fraction $\phi_{h,\tau}$ are depicted in Figure~\ref{fig:evophichns}, and corresponding pictures of the velocity are depicted in Figure~\ref{fig:evouchns} . For similar experiments see  \cite{Han2016,Han2015,LiShen2020}. 
%
\begin{figure}[htbp!]
\centering
\footnotesize
\begin{tabular}{ccc}
    \includegraphics[trim={3.3cm 1.4cm 1.5cm 0.5cm},clip,scale=0.32]{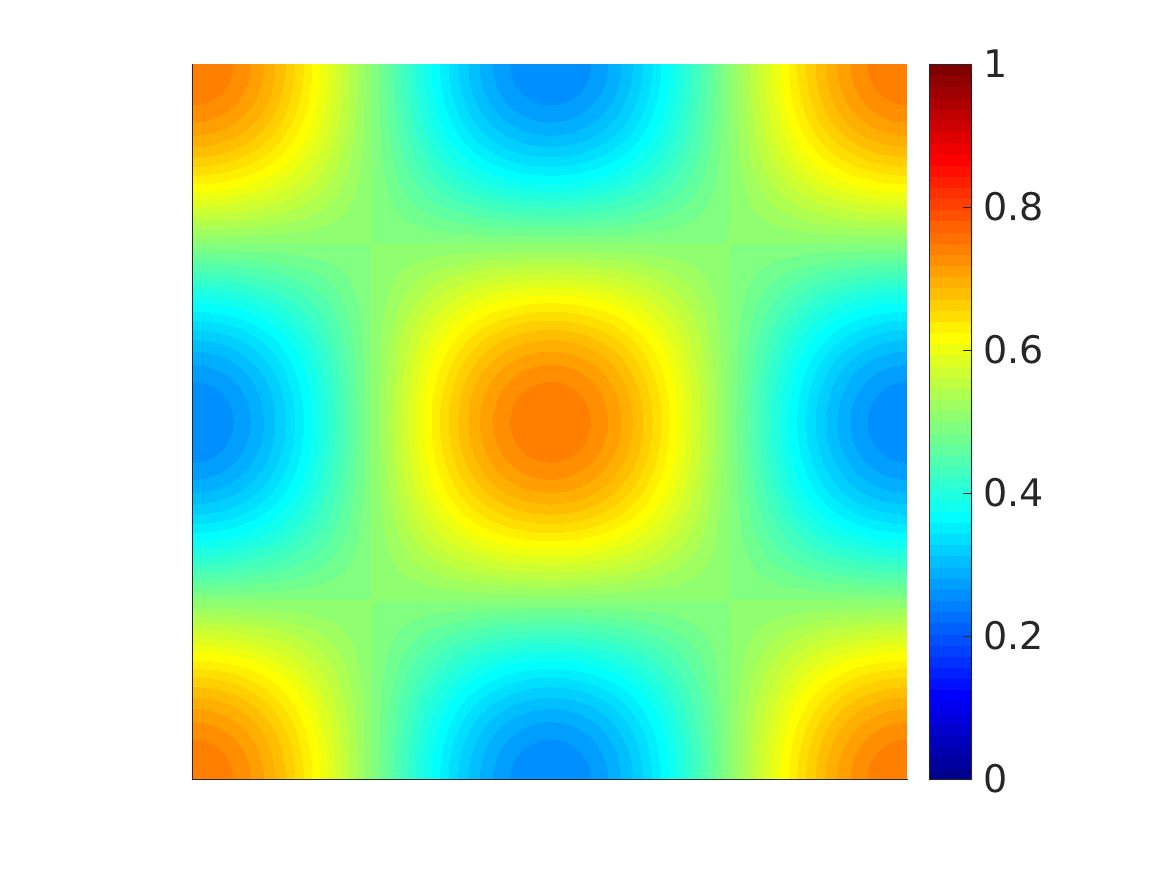} 
    &
    \includegraphics[trim={3.3cm 1.4cm 1.5cm 0.5cm},clip,scale=0.32]{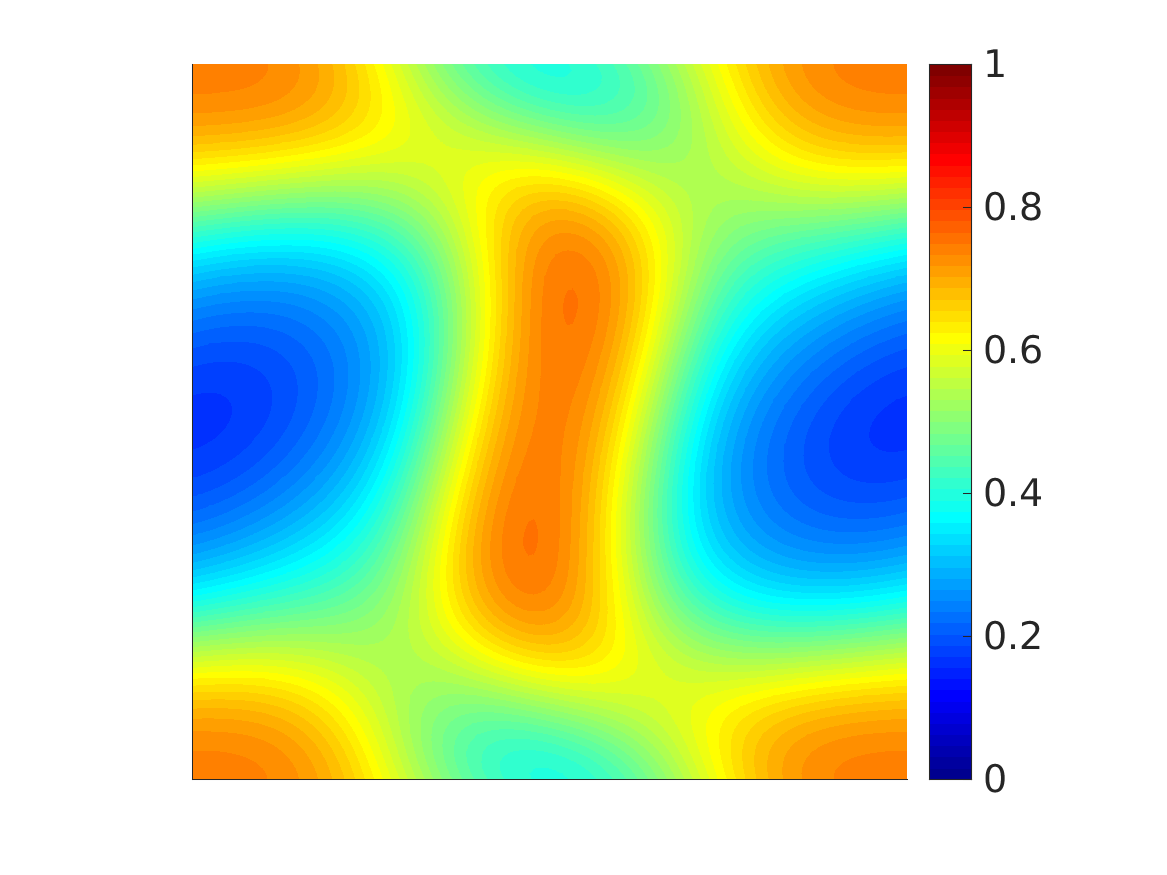}  
    &
    \includegraphics[trim={3.3cm 1.4cm 1.5cm 0.5cm},clip,scale=0.32]{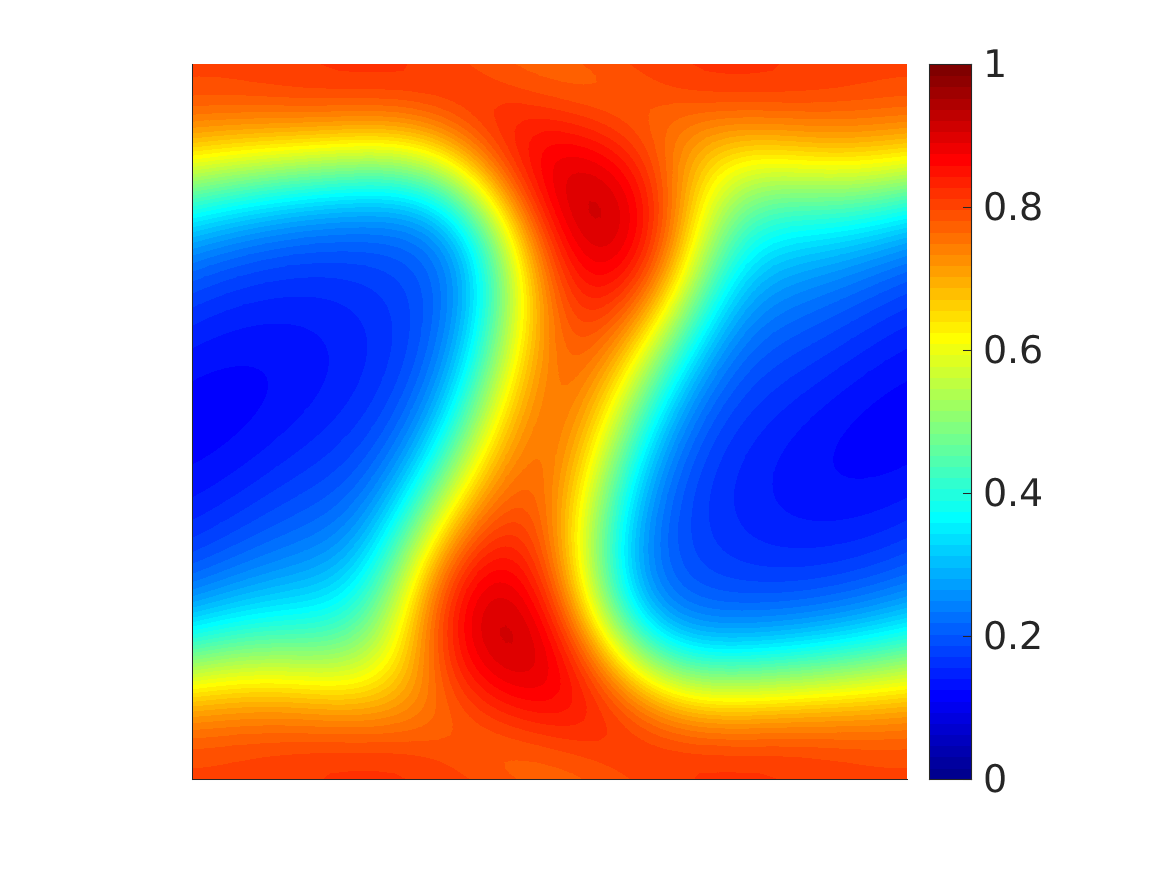}  \\[-0.5em]
    \hspace{-1.5em}t=0 & \hspace{-1.5em}t=0.4 & \hspace{-1.5em}t=0.8 \\
    \includegraphics[trim={3.3cm 1.4cm 1.5cm 0.5cm},clip,scale=0.32]{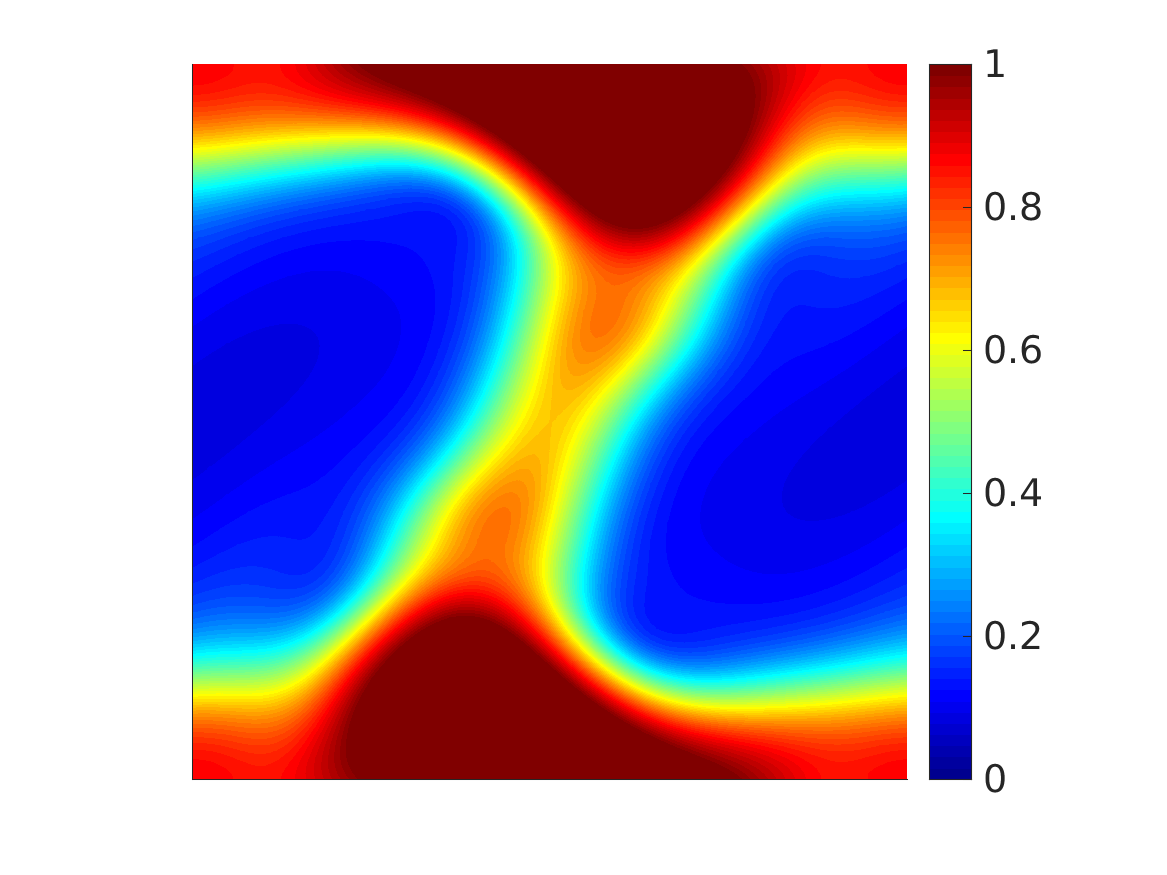}  
    &
    \includegraphics[trim={3.3cm 1.4cm 1.5cm 0.5cm},clip,scale=0.32]{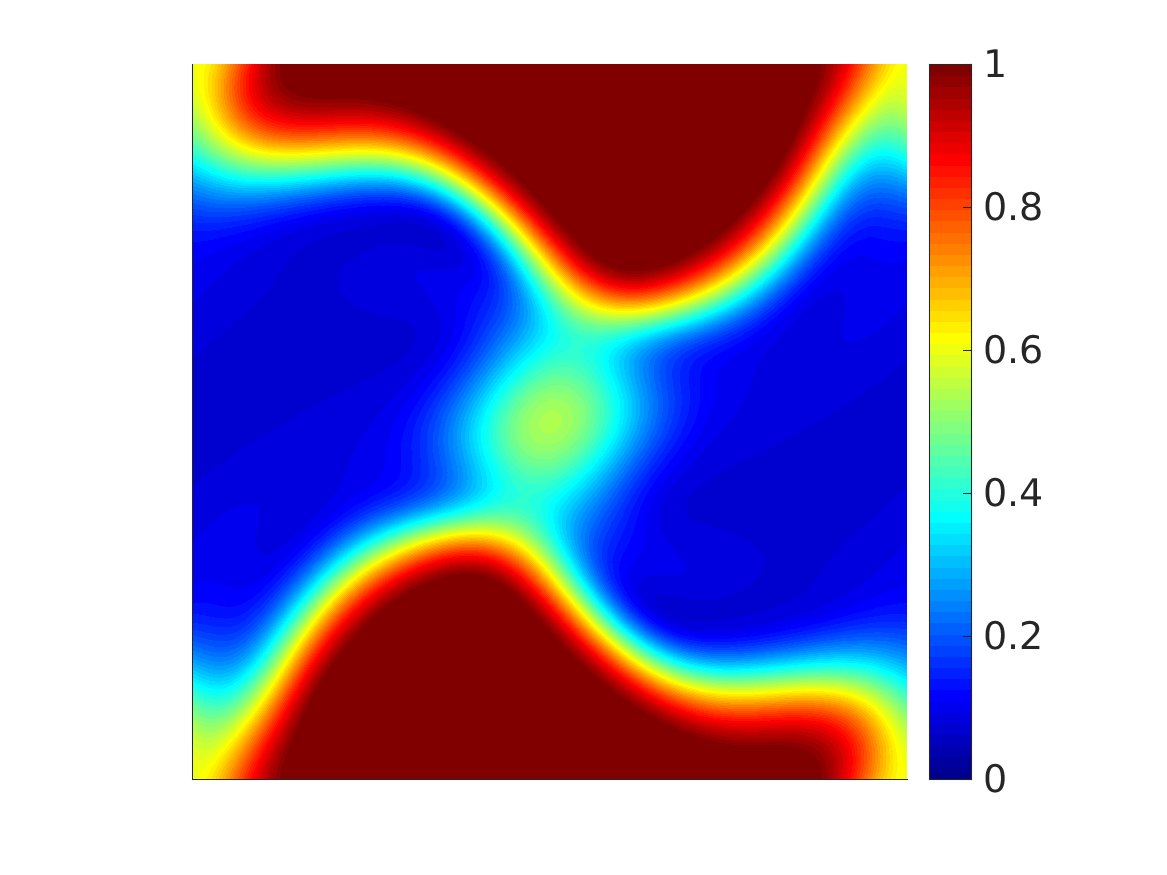} 
    &
    \includegraphics[trim={3.3cm 1.4cm 1.5cm 0.5cm},clip,scale=0.32]{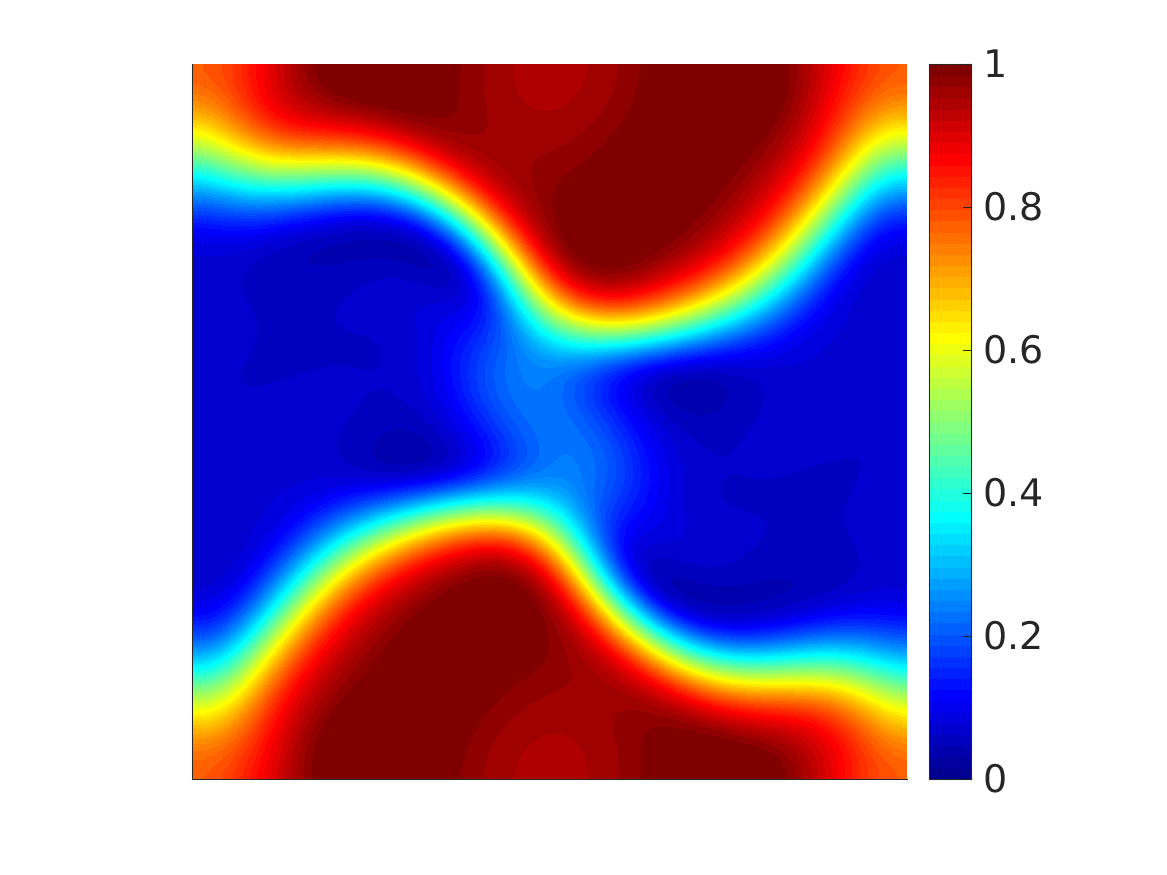} \\[-0.5em]
    \hspace{-1.5em}t=1.2 & \hspace{-1.5em}t=1.6 & \hspace{-1.5em}t=2 \\
\end{tabular}
    \caption{Snapshots of the volume fraction $\phi_{h,\tau}$ for $h=2^{-7}$ and $\tau=0.25 \cdot 2^{-7}$.\label{fig:evophichns}}
\end{figure}
\begin{figure}[htbp!]
\centering
\footnotesize
\begin{tabular}{ccc}
    \includegraphics[trim={3.3cm 1.4cm 1.5cm 0.0cm},clip,scale=0.32]{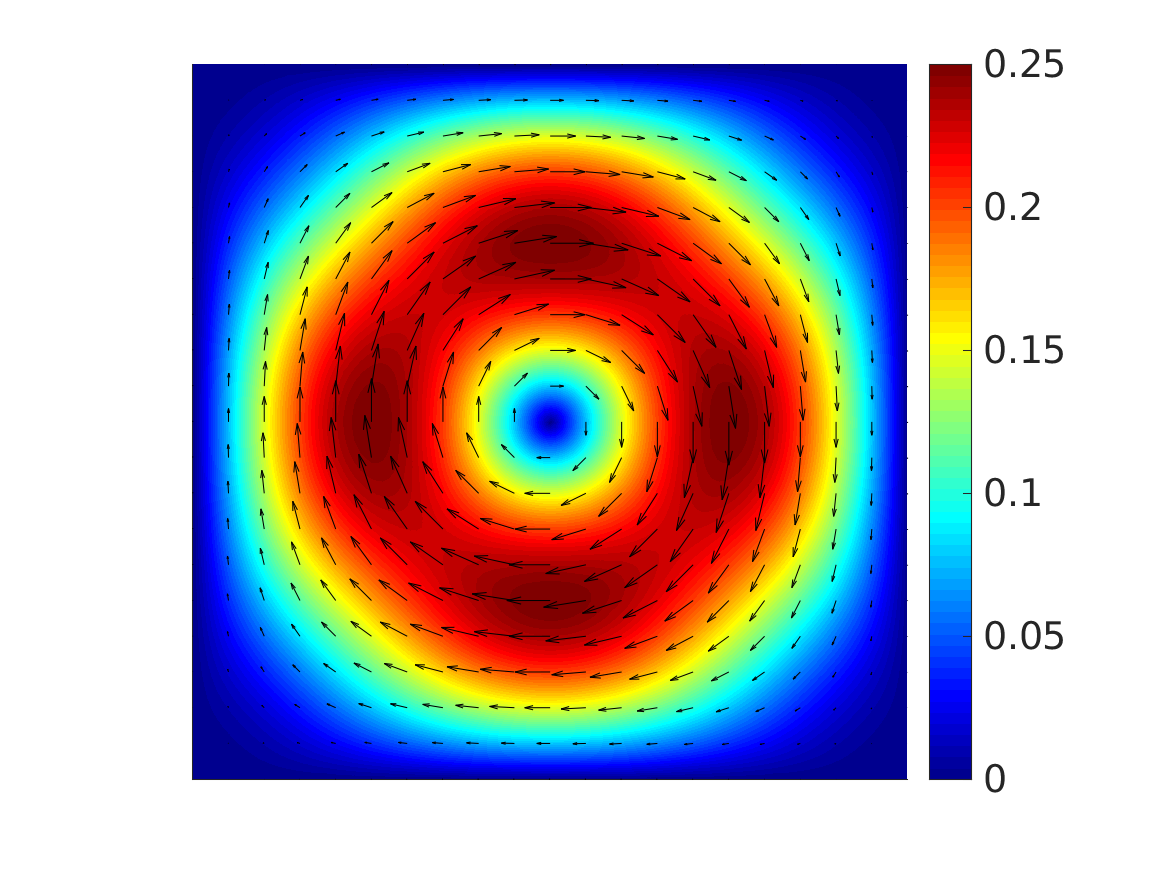} 
    &
    \includegraphics[trim={3.3cm 1.4cm 1.5cm 0.0cm},clip,scale=0.32]{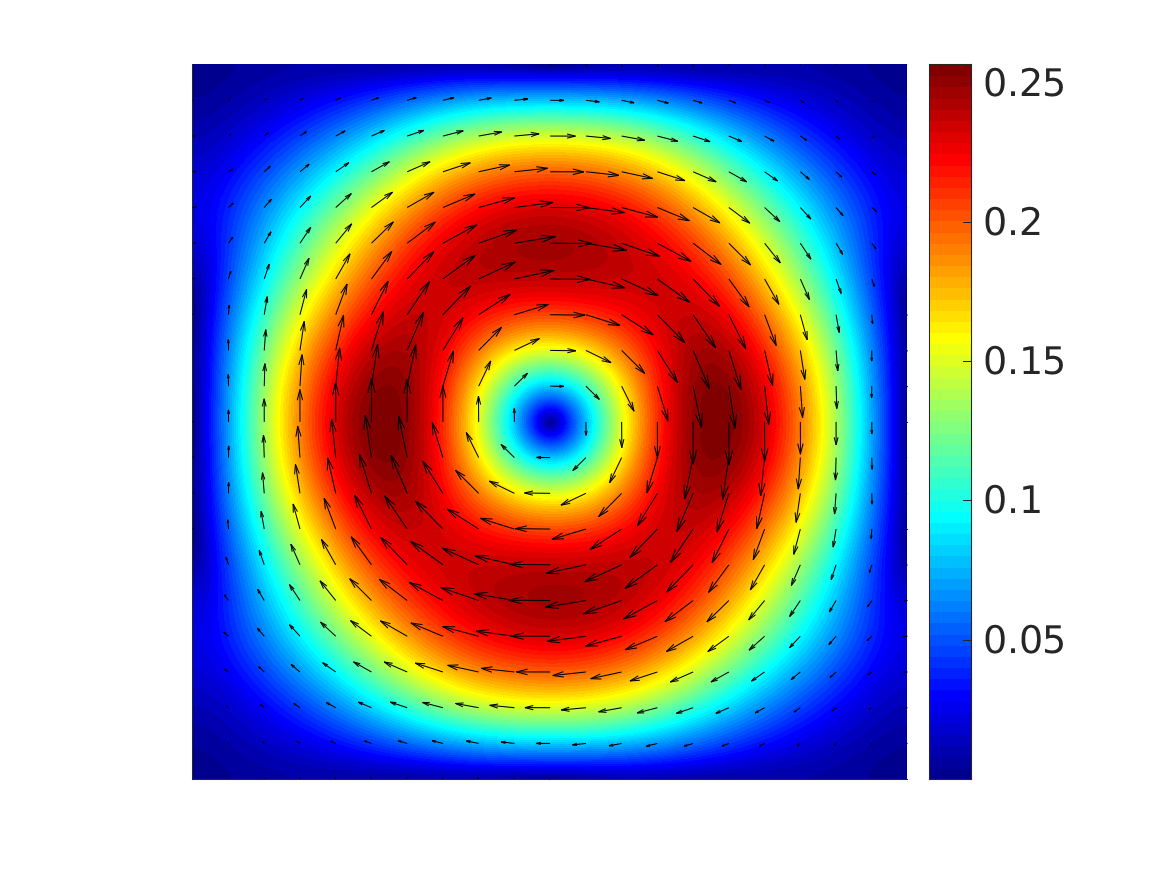}  
    &
    \includegraphics[trim={3.3cm 1.4cm 1.5cm 0.0cm},clip,scale=0.32]{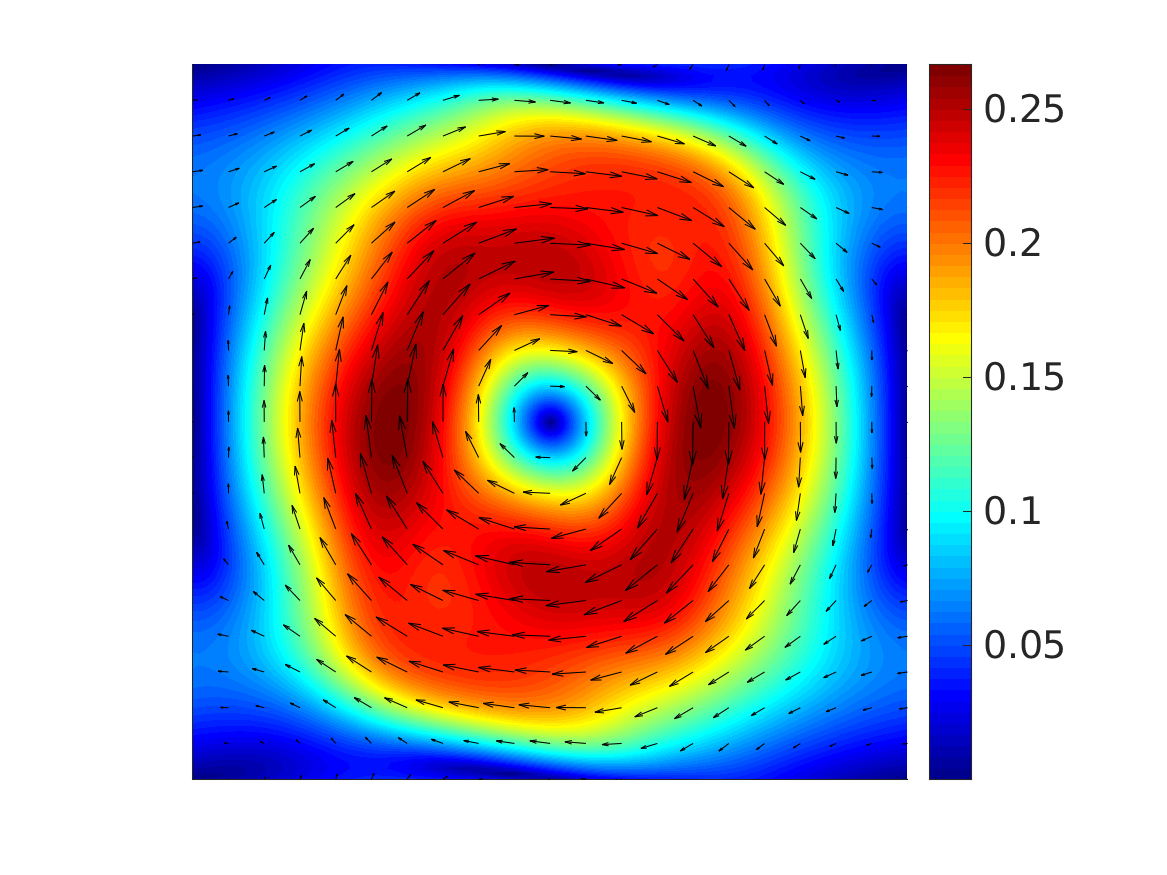} \\[-0.5em]
    \hspace{-1.5em}t=0 & \hspace{-1.5em}t=0.4 & \hspace{-1.5em}t=0.8 \\
    \includegraphics[trim={3.3cm 1.4cm 1.5cm 0.0cm},clip,scale=0.32]{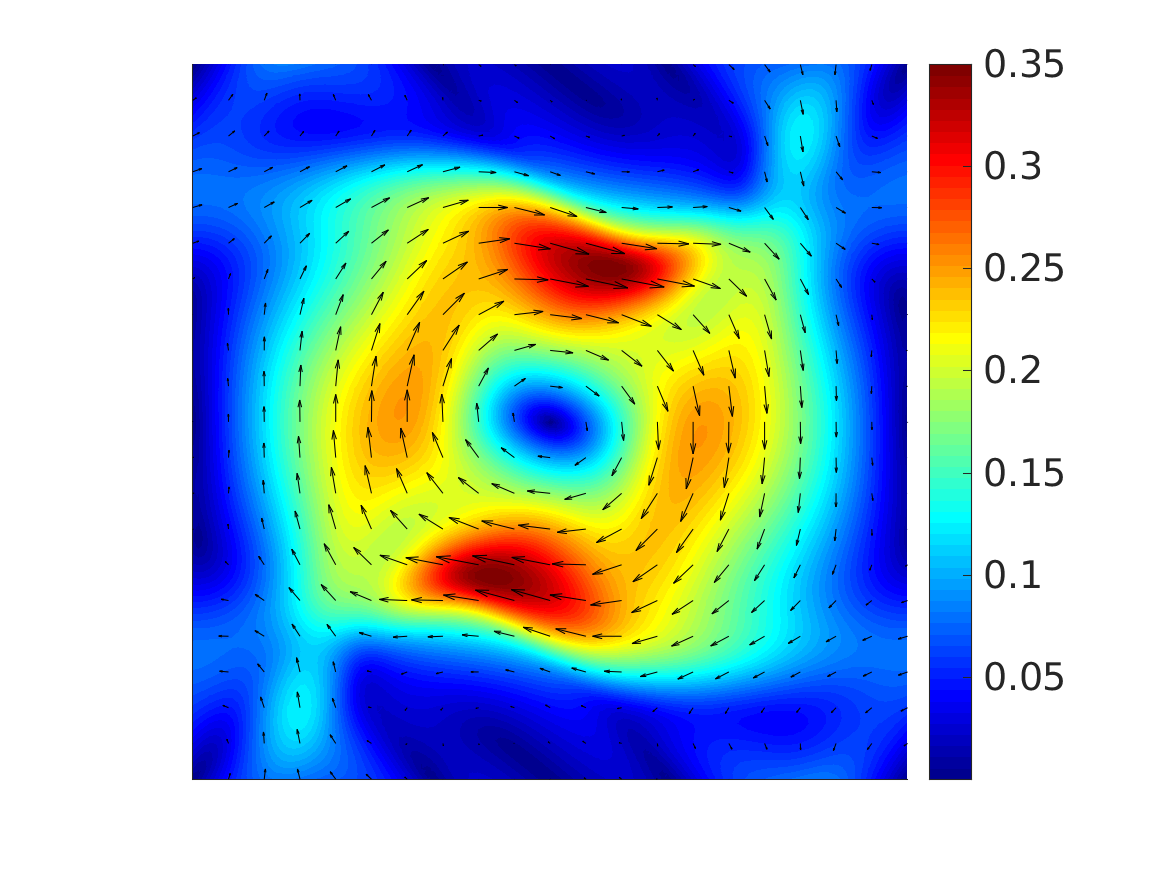}  
    &
    \includegraphics[trim={3.3cm 1.4cm 1.5cm 0.0cm},clip,scale=0.32]{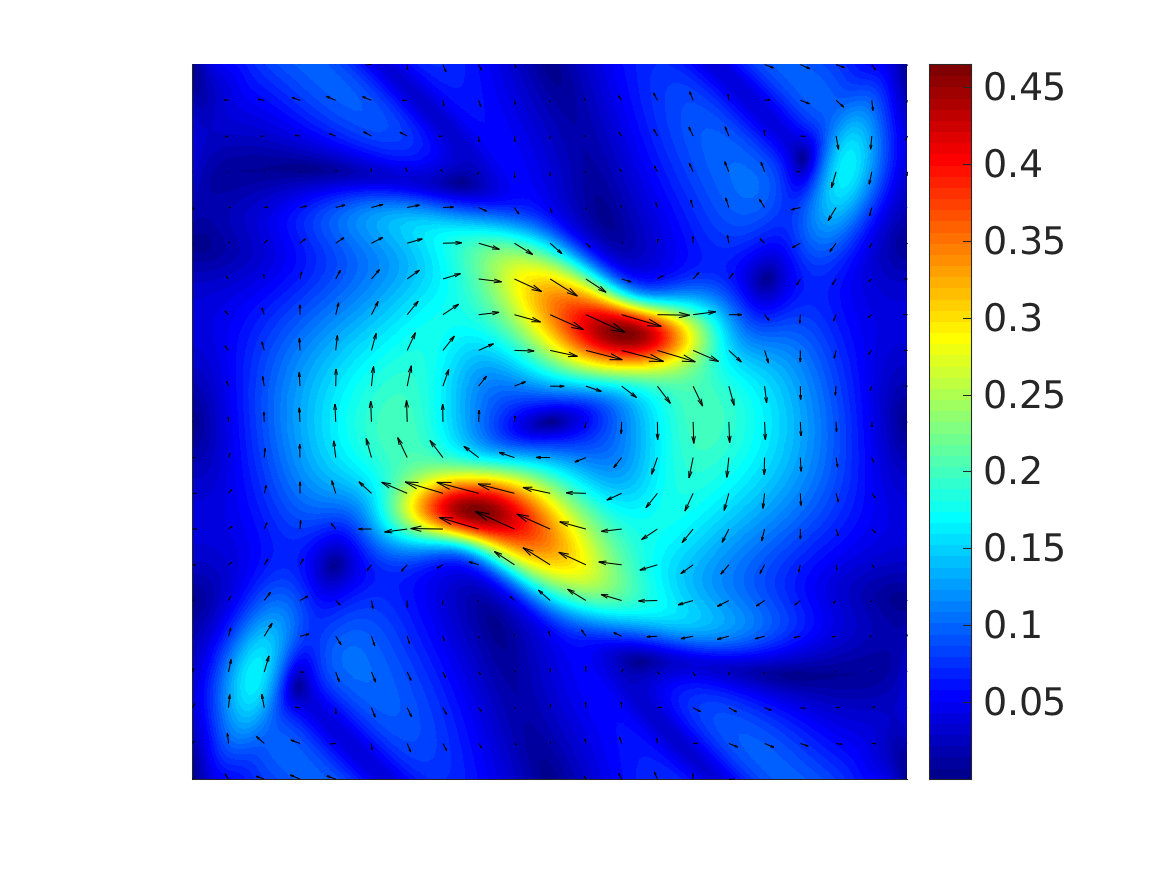} 
    &
    \includegraphics[trim={3.3cm 1.4cm 1.5cm 0.0cm},clip,scale=0.32]{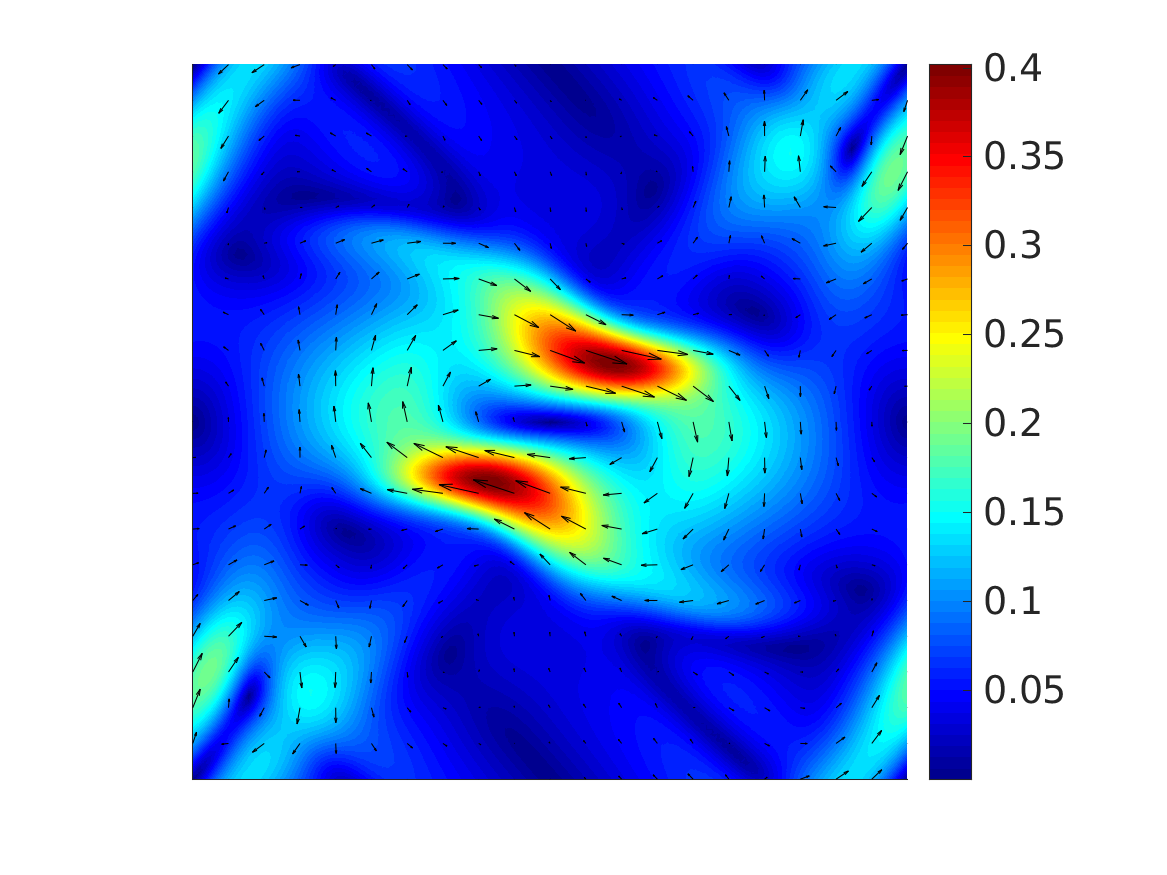} \\[-0.5em]
    \hspace{-1.5em}t=1.2 & \hspace{-1.5em}t=1.6 & \hspace{-1.5em}t=2 \\
\end{tabular}
    \caption{Snapshots of velocities $\snorm{\uu_{h,\tau}}^2$ for $h=2^{-7}$ and $\tau=0.25 \cdot 2^{-7}$.\label{fig:evouchns}}
\end{figure}
\begin{figure}[htbp!]
\centering
\footnotesize
\begin{tabular}{cc}
    \includegraphics[trim={0.4cm 0.4cm 1.5cm 0.0cm},clip,scale=0.41]{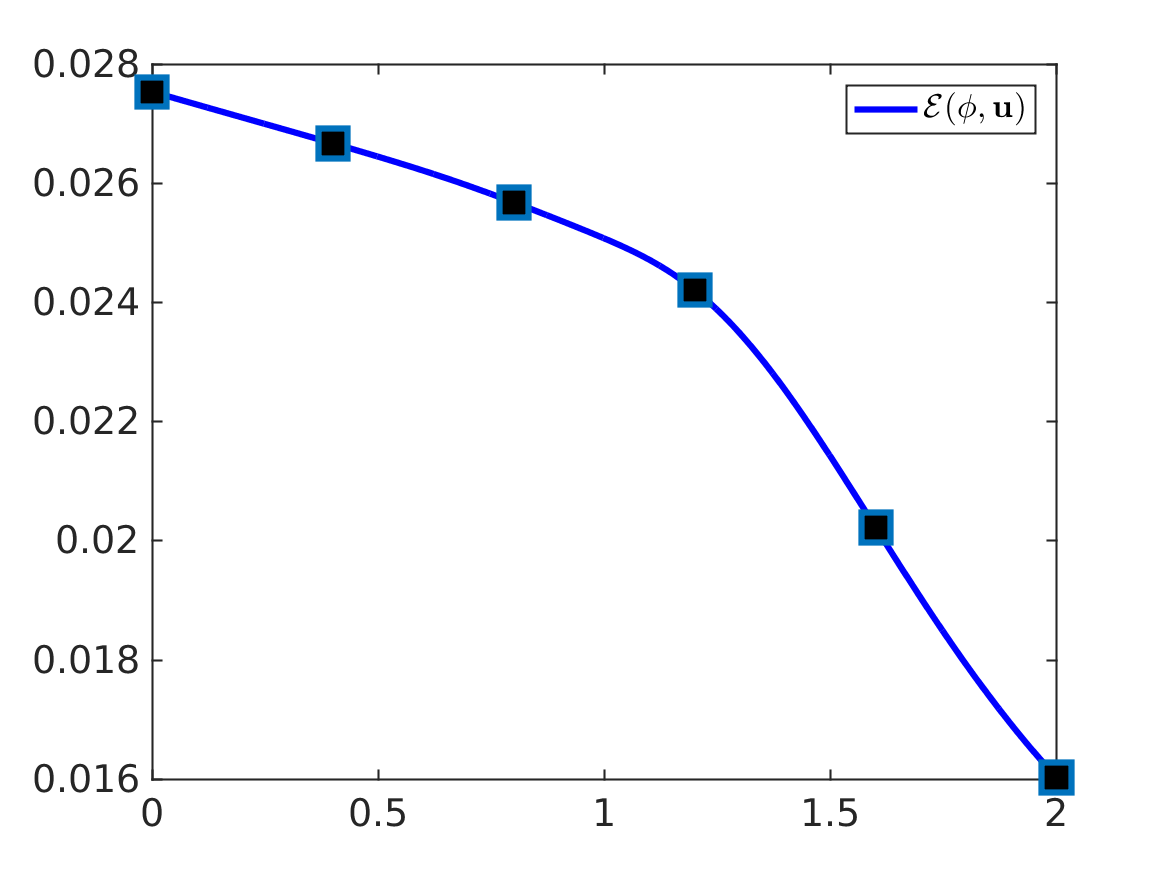} 
    &
    \includegraphics[trim={0.3cm 0.4cm 1.5cm 0.0cm},clip,scale=0.41]{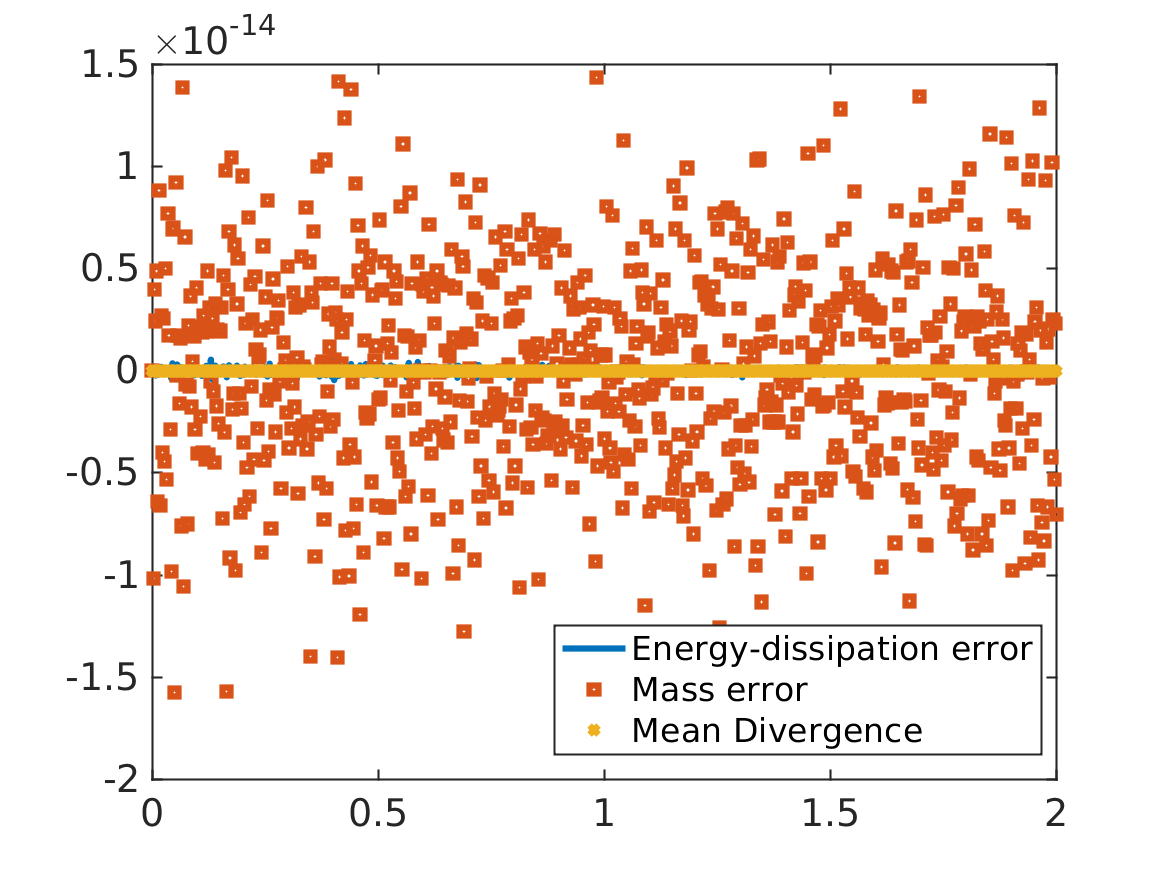}  
\end{tabular}
    \caption{Evolution of the energy $\E(\phi,\uu)$ (left) as well as energy-dissipation, mass conservation and mean divergence error (right) for simulations obtained with $h=2^{-7}$ and $\tau=0.25 \cdot 2^{-7}$.\label{fig:evochnsen}}
\end{figure} 
%
From these plots, one can deduce the following typical behaviour: 
At the beginning of the simulation, we observe a transition from the initial condition to an almost separated but connected distribution of the two phases.
At the same time, the rotating velocity starts to disperse, but remains  rather activate even until time $t=2$. 
We note that the energy of the system was monotonically decreasing over time in all our computations, while the total mass was preserved up to round-off accuracy.

\subsection*{Convergence results}

%
We now discuss the convergence rates observed in our computations. 
Since no analytical solution is available, the discretization error is estimated by comparing the computed solutions $(\phi_{h,\tau},\bar\mu_{h,\tau},\uu_{h,\tau},\bar p_{h,\tau})$ with those computed on uniformly refined grids.
The error quantities for the fully-discrete scheme, about which we report in the following, are defined as 
\begin{align*}
e_{h,\tau} &= \norm*{\phi_{h,\tau} - \phi_{h/2,\tau/2}}_{L^\infty(H^1)}^2 + \norm*{\uu_{h,\tau} - \uu_{h/2,\tau/2}}_{L^\infty(L^2)}^2 \\
&\qquad\quad+ \norm*{\bmu_{h,\tau} - \bmu_{h/2,\tau/2}}_{L^2(H^1)}^2 + \norm*{\bu_{h,\tau} - \bu_{h/2,\tau/2}}_{L^2(H^1)}^2 \\
e_{h,\tau}^p &=\norm*{\bar p_{h,\tau} - \bar p_{h/2,\tau/2}}_{L^2(L^2)}^2.
\end{align*}
%
%
In Table~\ref{tab:rates_time_chns}, we report about the results of our computations obtained on a sequence of uniformly refined meshes with mesh size $h_k=2^{-(3+k)}$, $k=0,\ldots,4$, and time steps $\tau_k = 0.025\cdot h_k$. 
As predicted by Theorem~\ref{thm:fulldisk}, we observe convergence of fourth order for the squared norms of the errors in all solution components. Let us note that these rates are optimal with respect to the approximation properties of the finite element spaces.
\begin{table}[htbp!]
\centering
\small
\caption{Errors and experimental orders of convergence for $h\approx \tau\approx 2^{-k}$. \label{tab:rates_time_chns}} 
\begin{tabular}{c||c|c||c|c}
$ k $ & $ e_{h,\tau} $  &  eoc & $e^p_{h,\tau}$ & eoc  \\
\hline
$ 0 $   & $5.748 \cdot 10^{-0}$  &    ---  & $4.863\cdot 10^{-4}$ &  ---\\
$ 1 $   & $1.523 \cdot 10^{-0}$  & 1.916   & $1.442\cdot 10^{-4}$ & 1.753 \\
$ 2 $   & $1.637 \cdot 10^{-1}$  & 3.217   & $1.586\cdot 10^{-5}$ & 3.184\\
$ 3 $   & $1.074 \cdot 10^{-2}$  & 3.929   & $5.544\cdot 10^{-7}$ & 4.839 
\end{tabular}
\end{table}
Following the suggestions in \cite{Kovacs19}, we also conducted convergence studies with respect to $\tau$ and $h$ separately, i.e., by fixing small $\tau$ and varying $h$, and vice versa. As predicted by Theorem~\ref{thm:fulldisk}, these tests lead to the same convergence rates in both discretization parameters.

\subsection*{Concluding remarks}
In summary, the proposed numerical scheme shows second order convergence in space and time and in all variables.

\section{Discussion} \label{sec:8}

In this paper, we proposed a structure-preserving variational discretization scheme for the Cahn-Hilliard Navier-Stokes system with concentration dependent mobility and viscosity, based on a mixed finite-element approximation in space and a Petrov-Galerkin discretization in time. 
The chosen discretization spaces are fully balanced and allow us to prove second order  convergence in all variables. 
The required regularity assumptions on the solution are weaker than in previous work, but still not minimal in view of the predicted convergence rates. To some extent, this can be explained by the lack of strong stability of Petrov-Galerkin time discretization; see \cite{AndreevSchweitzer2014} for details. 
At the core of our error analysis is a discrete stability estimate for the discrete problem, which was established by means of relative energy estimates. These naturally account for the nonlinearities in the energy of the problem and significantly simplified the error analysis.
Let us note that a continuous version of the respective stability estimate allows to establish stability and weak-strong uniqueness also for the continuous problem~\cite{brunk}.
The proposed variational time discretization approach, which is strongly related to the average vector field method \cite{Gonzales96,McLachlanEtAl99}, leads to a fully-implicit time stepping scheme which, however, can be solved rather efficiently. 
Typically 2--3 Newton iterations were sufficient to retain the full convergence order in our numerical tests. 
%

In this paper, a fully balanced second order approximation was considered in detail. The proposed variational discretization approach as well as its analysis, however, generalize quite naturally to higher order. 
The presented discrete stability analysis can further be used to simplify the derivation of error estimates for related, e.g., linearly implicit, time discretization schemes. This is accomplished by interpreting the discrete solution of a related scheme as perturbed solution of the proposed discretization scheme and estimating the resulting redicuals either a-priori or a-posteriori.
Detailed investigations in this direction as well as the extension to more complex multiphase problems, like the density dependent Cahn-Hilliard Navier-Stokes equations \cite{AGG} or non-isothermal phase-field models \cite{Penrose1990}, will be topics of future research. 

{\footnotesize

\section*{Acknowledgement}
Support by the German Science Foundation (DFG) via TRR~146 (project~C3) and SPP~2256 (project Eg-331/2-1) is gratefully acknowledged. M.L.\ is grateful to the Gutenberg Research College, University Mainz for supporting her research.}

\bibliographystyle{abbrv}
\bibliography{chns}

\newpage 


\appendix

\section*{Appendix}

For completeness, we now provide detailed proofs for some of the technical results that were used in the error analysis of the previous sections.

\section{Projection error estimates and Gronwall lemma} 
\label{app:proj}

In the following, we summarize some well-known results about standard projection and interpolation operators, which are used in our analysis. 

\subsection{Space discretization}
We consider the setting of Sections~\ref{sec:prelim} and \ref{sec:main} and, in particular, assume (A5)--(A6) to hold true.
%
%
The following results then follow with standard arguments; see e.g. \cite{BrennerScott}.
The $L^2$-orthogonal projection $\pi_h^0 : L^2(\Omega) \to \Vh$,
satisfies
%
\begin{align} \label{eq:l2projest}
    \|u - \pi_h^0 u\|_{H^s} \leq C h^{r-s} \|u\|_{H^r} \qquad \forall u \in H^r(\Omega),
\end{align}
and all parameters $-1 \le s \le r$ and $0 \le r \le 4$. 
On quasi-uniform meshes $\Th$, which we consider here, the projection $\pi_h^0$ is also stable with respect to the $H^1$-norm, i.e., 
\begin{align} \label{eq:h1stab_l2proj}
\|\pi_h^0 u\|_{H^1} \le C \|u\|_{H^1} \qquad \forall u \in H^1(\Omega).
\end{align}
%
For the $L^2$-projection $\tilde \pi^0_{h}:L^2_0(\Omega)\to \Qh$, 
one similarly has
\begin{align}
\label{eq:pressure_proj_est}
    \|q-\tilde \pi^0_{h} q\|_{L^2} \leq Ch^2\|q\|_{H^2} \qquad \forall q \in H^2(\Omega).
\end{align}
The $H^1$-elliptic projection $\pi_h^1 : H^1(\Omega) \to \Vh$, defined in  \eqref{eq:defh1proj}, satisfies
\begin{align}\label{eq:h1porjest}
 \|u - \pi_h^1 u\|_{H^s} \leq C h^{r-s} \|u\|_{H^r} \qquad \forall u \in H^r(\Omega),
\end{align}
for all parameters $-1 \le s \le r$ and $1 \le r \le 3$. 
Since we assumed quasi-uniformity of the mesh $\Th$, we can further resort to the inverse inequalities
\begin{align} \label{eq:inverse}
    \|v_h\|_{H^1} \le c_{inv} h^{-1} \|v_h\|_{L^2} 
    \qquad \text{and} \qquad 
    \|v_h\|_{L^p} \le c_{inv} h^{d/p-d/q} \|v_h\|_{L^q} 
\end{align}
which hold for all discrete functions $v_h \in \Vh$ and all $1 \le q \le p \le \infty$. 
These results also hold for vector valued functions.
%
%
%
%
The modified Stokes projector $\mathbf{P}^1_h: H^1(\Omega)^d \to \Zh$, defined in  \eqref{eq:defnmodstokes}, satisfies the error estimates
\begin{align}\label{eq:stokesproj}
    \|\uu-\mathbf{P}_h^1\uu\|_{H^s} \leq Ch^{r-s}\|\uu\|_{H^r} \qquad \forall \uu \in H^r(\Omega)^d
\end{align}
with $-1 \le s \le 1$ and $1 \le r \le 3$,
which follows by minor adaption of the results in~\cite{Guillngonzlez2012}. 

\subsection{Time discretization}
%
The piecewise linear interpolation 
$\I_\tau^1:H^1(0,T)\to P_1^c(\Itau)$ and the piecewise constant projection $\bar \pi_\tau^0 : L^2(0,T) \to P_0(\I_\tau)$ in time satify
\begin{align} 
    \|u - \bar\pi_\tau^0 u\|_{L^p(0,T)} &\le C \tau^{1/p-1/q+r} \|u\|_{W^{r,q}(0,T)} \quad && \forall u \in W^{r,q}(0,T), \label{eq:timprojest}\\
    \|u - \I_\tau^1 u\|_{L^p(0,T)} &\leq C\tau^{1/p-1/q+2} \|u\|_{W^{r,q}(0,T)} \quad && \forall u \in W^{s,q}(0,T) \label{eq:timinterpest}
\end{align}
with $1 \le p \le q  \le \infty$ and for $0 \le r \le 1$ respectively $1 \le s \le 2$.
Moreover, these operator commute with differentiation in the sense that
\begin{align} \label{eq:commuting} 
\dt (\I_\tau^1 u) = \bar \pi_\tau^0 (\dt u).
\end{align}

\subsection{Projection estimates for nonlinear terms}

We now derive an estimate for the projection error of products of functions in time. 
Let $J=(a,b)$ be an open interval, and denote by $\pi_k : L^2(J) \to P_k(J)$ the $L^2$-orthogonal projection and by $I_k : H^1(J) \to P_k(J)$ the Lagrange interpolation operator with respect to the Gau{\ss}-Legendre quadrature points, i.e., the zeros of the $(k+1)$st Legendre Polynomial $\ell_{k+1}(t)$, which is scaled to the respective interval.
We start with an auxiliary result.
\begin{lemma} \label{lem:super}
Let $a \in W^{j+2,p}(J)$ with $j \le k$ and $1 \le p \le \infty$. Then 
\begin{align} \label{eq:super}
    \|\pi_k a  - I_k a\|_{L^p(J)} \le C |J|^{j+2} \|a\|_{W^{j+2,p}(J)}, \qquad j \le k.
\end{align}
Here $|J| = b-a$ and $C$ depends only on the polynomial degree $k$.
\end{lemma}
\noindent 
The result is well-known, but for completeness, we include a short proof.
\begin{proof}
Since $\pi_k a = a = I_k a$ for $a \in P_k$, we have $\pi_k a - I_k a = 0$ for $a \in P_k(J)$. 
Next consider $a = \ell_{k+1}$. By orthogonality of the Legendre polynomials, we have $\pi_k \ell_{k+1} = 0$ and due to the choice of interpolation points, also $I_k\ell_{k+1} = 0$. 
Due to the linearity of the projection and interpolation operator, we thus have $\pi_k a = I_k a$ for all $a \in P_{k+1}(J)$. 
The result then follows from the Bramble-Hilbert lemma and the mapping trick; see e.g. \cite{BrennerScott}.
\end{proof}

We can now establish the following projection error estimate.
\begin{lemma} \label{lem:average_time_err}
Let $\bar a = \pi_k a$ denote the $L^2$-orthogonal projection onto $P_{0}(\Itau)$.  
Then for any $u,v \in W^{2,p}(0,T)$ with $1 \le p \le \infty$, one has 
\begin{align}
\|\overline{\bar u\bar v}-\overline{uv}\|_{L^p(0,T)} 
&\le C \tau^{2} \|u\|_{W^{2,p}(0,T)} \|v\|_{W^{2,p}(0,T)},  \label{eq:midpoint_order}
\end{align}
with a constant $C$ depending only on the polynomial degree $k$. 
\end{lemma}
\noindent 
\begin{proof}
As a first step, we show that for every $J=(t^{n-1},t^n)$ and $j \le k$, one can bound
\begin{align} \label{eq:est1}
\|\overline{\bar u \bar v} - \overline{uv}\|_{L^p} \le C \tau^{j+2} (\|u\|_{W^{j+2,p}} \|v\|_{W^{j+1,\infty}} + \|u\|_{W^{j+1,\infty}} \|v\|_{W^{j+2,p}})
\end{align}
where $\bar a = \pi_k a$ is the $L^2$ projection and $\|\cdot\|_{W^{k,p}}$ refers to the norm on this element $J$. 
We further write $\tilde a = I_k a$ for the interpolation and observe that 
\begin{align*}
\|\overline{\bar u \bar v} - \overline{uv}\|_{L^p} 
   &\le \|\overline{\bar u \bar v} - \overline{\tilde u \bar v}\|_{L^p}
 + \|\overline{\tilde u \bar v} - \overline{\tilde u \tilde v}\|_{L^p}
      + \|\overline{\tilde u \tilde v} - \widetilde{\tilde u \tilde v}\|_{L^p} \\
&\qquad 
+ \|\widetilde{\tilde u \tilde v} - \widetilde{u v}\|_{L^p}
      + \|\widetilde{u v} - \overline{u v}\|_{L^p}  
= (i) + (ii) + (iii) + (iv) + (v).      
\end{align*}
We can now use Lemma~\ref{lem:super} to estimate the individual terms. 
For ease of notation, we write $a \lesssim b$ for $a \le C b$ in the following, where $C$ may depend on the polynomial degree $k$.
For the first term, we have
\begin{align*}
(i) \lesssim  \|\bar u \bar v - \tilde u \bar v\|_{L^p} 
\le \|\bar u - \tilde u\|_{L^p} \|\bar v\|_{L^\infty} \lesssim \tau^{j+2} \|u\|_{W^{j+2,p}} \|\bar v\|_{L^\infty}.
\end{align*}
In a similar manner, we can estimate the second term by 
$(ii) \lesssim \tau^{j+2} \|u\|_{L^\infty} \|v\|_{W^{j+2,p}}$.
By Lemma~\ref{lem:super}, we can estimate the third term by 
\begin{align*}
(iii) \lesssim \tau^{j+2} \|\tilde u \tilde v\|_{W^{j+2,p}} 
\lesssim \tau^{j+2} (\|u\|_{W^{j+2,p}} \|v \|_{W^{j+1,\infty}} + \|u\|_{W^{j+1,\infty}} \|v \|_{W^{j+2,p}}),  
\end{align*}
which could still be refined a bit concerning the regularity requirements. 
The fourth term vanishes identically, i.e, $(iv)=0$, and the last is simply estimated by 
\begin{align*}
(v) \lesssim \tau^{j+2} \|uv\|_{W^{j+2^,p}} \lesssim \tau^{j+2} (\|u\|_{W^{j+2,p}} \|v \|_{W^{j+1,\infty}} + \|u\|_{W^{j+1,\infty}} \|v \|_{W^{j+2,p}}),  
\end{align*}
where we used Lemma~\ref{lem:super} once more. 
The assertion of Lemma~\ref{lem:average_time_err} now follows by summing up the $p$th power of the estimates \eqref{eq:est1} for the individual elements and noting that $\|a\|_{W^{j+1,\infty}(0,T)} \le C \|a\|_{W^{j+2,p}(0,T)}$ by continuous embedding. 
\end{proof} 

In a similar manner, we obtain the following estimate.
\begin{lemma}
Let $\bar a = \pi_0 a$ denote the $L^2$-orthogonal projection onto $P_{0}(\Itau)$. Furthermore, let $\phi\in P_1(\Itau)$. 
Then for any $u,v \in W^{2,p}(0,T)$ with $1 \le p \le \infty$, one has 
\begin{align}
\| g(\bar\phi)-\overline{g(\phi)}\|_{L^p(0,T)} 
&\le C \tau^{2} \|g(\phi)\|_{W^{2,p}(0,T)},  \label{eq:midpoint_order_single}
\end{align}
with a constant $C$ depending only on the polynomial degree $k$. 
\end{lemma}
\begin{proof}
We denote by $\tilde a=a(t^{n-1/2})$ and estimate
\begin{align*}
 \| g(\bar\phi)-\overline{g(\phi)}\|_{L^p} 
 \leq \| g(\bar\phi)-\widetilde{g(\phi)}\|_{L^p}  + \|\widetilde{ g(\phi)}-\overline{g(\phi)}\|_{L^p}    
\end{align*}
The first term can be written as
\begin{align*}
\| g(\bar\phi)-\widetilde{g(\phi)}\|_{L^p} 
&= \| g(\tilde\phi)-\widetilde{g(\phi)}\|_{L^p} \\
&= \|g(\phi(t^{n-1/2})) - g(\phi(t^{n-1/2}))\|_{L^p} =0,  
\end{align*}
and the second term is already of second order; see the proof of the previous lemma. 
\end{proof}

\subsection{Discrete Gronwall lemma}

In our stability analysis, we also employ the following well-known argument, whose proof follows immediately by induction.
\begin{lemma} \label{lem:discgronwall}
Let $(a^n)_n$, $(b_n)_n$, $(c_n)_n$, and $(\lambda_n)_n$ be given positive sequences, satisfying
\begin{align*}
    u_n + b_n \le e^{\lambda_n} u_{n-1} + c_n, \qquad n \ge 0.
\end{align*}
Then 
\begin{align} \label{eq:discgronwall}
    u_n + \sum_{k=1}^n e^{\sum_{j={k+1}}^{n} \lambda_j} b_k 
    \le e^{\sum_{j=1}^n \lambda_j} u_0 + \sum_{k=0}^n e^{\sum_{j={k+1}}^{n} \lambda_j} c_k, \quad n > 0. 
\end{align}
\end{lemma}

\section{Proof of Lemma~\ref{lem:fullstab}}
\label{app:fullstab}
%
Under the assumptions (A0)--(A7), we now establish the discrete stability estimate
\begin{align*}
 \E_\alpha(\phi_{h,\tau},&\uu_{h,\tau}|\hat \phi_{h,\tau},\hat\uu_{h,\tau}) \, \Big|_{t^{n-1}}^{t^n} + \frac{3}{4} \int_{t^{n-1}}^{t^n} \D_{\bar \phi_{h,\tau}}(\bmu_{h,\tau}-\hbmu_{h,\tau},  \bu_{h,\tau}-\hbu_{h,\tau}) \, ds 
 \\
 & \leq \bar c  \int_{t^{n-1}}^{t^n} \E_\alpha(\phi_{h,\tau},\uu_{h,\tau}|\hat \phi_{h,\tau},\hat\uu_{h,\tau})\; ds \\ 
 & \qquad \qquad \qquad + \bar C \int_{t^{n-1}}^{t^{n}} \|\bar r_{1,h,\tau}\|_{H^{-1}}^2 + \|\bar r_{2,h,\tau}\|_{H^1}^2 + \|\bar \rr_{3,h,\tau}\|_{\Zh^*}^2 \, ds,
\end{align*}
Let us recall the definition $f(z|\hat z) = f(z) - f(\hat z) - \langle f'(\hat z), z - \hat z \rangle$, with $\langle a, b \rangle$ denoting the Euclidean scalar product. 
Then for a smooth function $z : \RR \to \RR^n$, $t \mapsto z(t)$, one has 
\begin{align} \label{eq:aux}
\frac{d}{dt} f(z|\hat z) 
&= \la f'(z), \dt z\ra - \la f'(\hat z), \dt \hat z\ra -  
   \la f'(\hat z), \dt z - \dt \hat z \ra - \la f''(\hat z) \dt \hat z, z - \hat z\ra \nonumber\\
&= \la f'(z) - f'(\hat z), \dt z - \dt \hat z \ra + 
   \la f'(z) - f'(\hat z) - f''(\hat z) (z-\hat z), \dt \hat z  \ra. 
\end{align}
This result follows immediately by the chain-rule of differentiation and some elementary manipulations.
To simplify notation, we further split the relative energy functional $\E_\alpha(\phi,\uu|\hat \phi,\hat \uu)$ into the two independent contributions $\E_\alpha^\phi(\phi|\hat \phi)=\E_\alpha(\phi,0|\hat \phi,0)$ and $\E_\alpha^\uu(\uu|\hat \uu)=\E_\alpha(0,\uu|0,\hat \uu)$. 
Then by the fundamental theorem of calculus, we can see that
\begin{align} \label{eq:relenest}
\E_\alpha(\phi_{h,\tau},&\uu_{h,\tau}|\hat\phi_{h,\tau},\hat\uu_{h,\tau}) \, \Big|_{t^{n-1}}^{t^n}
 = \int_{t^{n-1}}^{t^n} \frac{d}{dt}\E_\alpha(\phi_{h,\tau},\uu_{h,\tau}|\hat\phi_{h,\tau},\hat\uu_{h,\tau}) \, ds  \\
&= \int_{t^{n-1}}^{t^n} \frac{d}{dt}\E_\alpha^\phi(\phi_{h,\tau}|\hat\phi_{h,\tau}) \, ds + 
\int_{t^{n-1}}^{t^n}
\frac{d}{dt}\E_\alpha^\uu(\uu_{h,\tau}|\hat\uu_{h,\tau}) \,  ds
= (I) + (II). \notag
\end{align}
The two terms $(I)$ and $(II)$ can now be treated separately.

\subsection*{First term}
By applying the identity \eqref{eq:aux} to the integrand in the first term (i), and then employing the variational identities \eqref{eq:pg2} and \eqref{eq:disc_pert2}, we obtain
\begin{align*}
(I) &=  \gamma \la \nabla \phi_{h,\tau} - \nabla \hat \phi_{h,\tau}, \nabla \dt \phi_{h,\tau} - \nabla \dt \hat \phi_{h,\tau}\ra^n  +\la f'(\phi_{h,\tau}) - f'(\hat\phi_{h,\tau}), \dt\phi_{h,\tau} - \dt\hat\phi_{h,\tau} \ra^n \\
     & \qquad \qquad + \la f'(\phi_{h,\tau}) - f'(\hat \phi_{h,\tau}) - f''(\hat \phi_{h,\tau}) (\phi_{h,\tau} - \hat \phi_{h,\tau}), \dt \hat \phi_{h,\tau}\ra^n \\
     & \qquad \qquad + \alpha \la \phi_{h,\tau} - \hat \phi_{h,\tau}, \dt \phi_{h,\tau} - \dt \hat \phi_{h,\tau}\ra^n   
\\
    &= \la \bmu_{h,\tau} - \hbmu_{h,\tau} + \bar r_{2,h,\tau}, \dt \phi_{h,\tau} - \dt \hat \phi_{h,\tau}\ra^n 
   + \alpha \la \phi_{h,\tau}- \hat \phi_{h,\tau}, \dt \phi_{h,\tau} - \dt \hat \phi_{h,\tau}\ra^n \\ 
   &\qquad \qquad 
   + \la f'(\phi_{h,\tau}) - f'(\hat \phi_{h,\tau}) - f''(\hat \phi_{h,\tau}) (\phi_{h,\tau} - \hat \phi_{h,\tau}), \dt \hat \phi_{h,\tau}\ra^n.
\end{align*}
Here $\bar \xi_{h,\tau} = \dt \phi_{h,\tau} - \dt \hat \phi_{h,\tau}\in P_0(\Itau;\Vh)$ was used as an admissible test function in \eqref{eq:pg2} and \eqref{eq:disc_pert2}.
Since this function is piecewise constant in time, we can replace 
\begin{align}
\alpha \la \phi_{h,\tau}- \hat \phi_{h,\tau}, \dt \phi_{h,\tau} - \dt \hat \phi_{h,\tau}\ra^n
= 
\alpha \la \bar \phi_{h,\tau}- \hat{\bar\phi}_{h,\tau}, \dt\phi_{h,\tau} - \dt \hat \phi_{h,\tau}\ra^n. \label{eq:l2tovsnormal}
\end{align}
Recall that $\bar g = \bar\pi_\tau^0 g$ is used to abbreviate the projection onto piecewise constants in time.  
By employing $\bar \psi_{h,\tau} = \mu_{h,\tau} - \hat \mu_{h,\tau} + \bar r_{2,h,\tau} + \alpha \bar\pi_\tau^0 (\phi_{h,\tau} - \hat \phi_{h,\tau}) \in P_0(\Itau;\Vh)$ as a test function in the identities \eqref{eq:pg1} and \eqref{eq:disc_pert1}, we further obtain
\begin{align*}
(I) &= 
-\la b(\bar\phi_{h,\tau}) \nabla (\bmu_{h,\tau} - \hbmu_{h,\tau}), \nabla(\bmu_{h,\tau}-\hbmu_{h,\tau} + \bar r_{2,h,\tau})\ra^n 
- \la \bar r_{1,h,\tau}, \bmu_{h,\tau} - \hbmu_{h,\tau} + \bar r_{2,h,\tau} \ra^n 
\\
 &  \qquad \qquad
 +\la \phi_{h,\tau}(\bu_{h,\tau}-\hbu_{h,\tau}),\nabla(\bmu_{h,\tau}-\hbmu_{h,\tau}+\hat r_{2,h,\tau})\ra^n \\
& \qquad \qquad  
- \alpha \la b(\bar\phi_{h,\tau}) \nabla (\bmu_{h,\tau} - \hbmu_{h,\tau}), \nabla (\phi_{h,\tau} - \hat \phi_{h,\tau}) \ra^n 
- \alpha \la \bar r_{1,h,\tau}, \phi_{h,\tau} - \hat \phi_{h,\tau} \ra^n 
\\
 & \qquad \qquad  
+\alpha \la \phi_{h,\tau}(\bu_{h,\tau}-\hbu_{h,\tau}),\nabla(\phi_{h,\tau}-\hat\phi_{h,\tau})\ra^n \\
 & \qquad \qquad  
+ \la f'(\phi_{h,\tau}) - f'(\hat \phi_{h,\tau}) - f''(\hat \phi_{h,\tau})(\phi_{h,\tau} - \hat \phi_{h,\tau}), \dt \hat \phi_{h,\tau} \ra^n
\\
 &= (i) + (ii) + (iii) + (iv) + (v) + (vi) + (vii).
\end{align*}
Via Young's inequality, the first term can be estimated by
\begin{align*}
(i) 
& \leq  -(1-\tfrac{1}{8})   \int_{t^{n-1}}^{t^n} \D^\mu_{\bar \phi_{h,\tau}}(\bmu_{h,\tau}-\hbmu_{h,\tau})\, ds + C_1 \int_{t^{n-1}}^{t^n} \|\bar r_{2,h,\tau}\|_{H^1}^2 \, ds,
\end{align*}
where $\D^\mu_\phi(\mu)=\D_\phi(\mu,0)$ is the $\mu$-dependent part of the dissipation functional. 
%
By Cauchy-Schwarz and Young inequalities, the second term can be estimated by
\begin{align*}
(ii) &\leq 
\int_{t_{n-1}}^{t^n} (1+\tfrac{1}{4\delta}) \|\bar r_{1,h,\tau}\|_{H^{-1}}^2 + \|\bar r_{2,h,\tau}\|^2_{H^1} + \delta \|\bmu_{h,\tau} - \hbmu_{h,\tau}\|^2_{H^1} \,ds.
\end{align*}
By testing \eqref{eq:pg2} and \eqref{eq:disc_pert4} with $\bar \xi_{h,\tau}=1$, and noting that all functions are independent of time, we see that 
\begin{align}
\la \bmu_{h,\tau} - \hbmu_{h,\tau},1\ra
&= \gamma \la \nabla \bar \phi_{h,\tau} - \nabla \hbphi_{h,\tau}, \nabla 1\ra + \la \overline{f'(\phi_{h,\tau})} - \overline{f'(\hat \phi_{h,\tau})}, 1\ra + \la \bar r_{2,h,\tau},1\ra \notag\\
&\le C_2' \| \phi_{h,\tau} -  \hat \phi_{h,\tau}\|_{L^\infty(H^1)} +  C_\Omega\norm*{\bar r_{2,h,\tau}}_{L^2}, \label{eq:meanmuerr}
\end{align}
for almost every point in time, with  $C_2'$ depending on $\gamma$ and the uniform bounds for $\phi_{h,\tau}$ and $\hat \phi_{h,\tau}$ in the $L^\infty(H^1)$-norm. 
By a Poincar\'e inequality and \eqref{eq:meanmuerr} we can thus estimate 
\begin{align} \label{eq:muest}
\|\bmu_{h,\tau} - \hbmu_{h,\tau}\|_{H^1}^2 
&\le C_P (\|\nabla (\bmu_{h,\tau} - \hbmu_{h,\tau}) \|_{L^2}^2 + |\la \bmu_{h,\tau} - \hbmu_{h,\tau},1\ra|^2 \\
&\le C_P \|\nabla (\bmu_{h,\tau} - \hbmu_{h,\tau}) \|_{L^2}^2 + C_P' \|\phi_{h,\tau} - \hat \phi_{h,\tau}\|_{H^1}^2 +  C_\Omega\norm*{\bar r_{2,h,\tau}}_{L^2}^2. \notag
\end{align}
The two remaining terms can further be estimated by the dissipation and relative energy. 
After choosing $\delta$ sufficiently small, we finally obtain
\begin{align*}
(ii) &\le \frac{1}{8} \int_{t^{n-1}}^{t^n} \D^\mu_{\bar \phi_{h,\tau}}(\bmu_{h,\tau}-\hbmu_{h,\tau}) \,ds 
+ C_2 \int_{t^{n-1}}^{t^n} \|\bar r_{1,h,\tau}\|_{H^{-1}}^2 + \|\bar r_{2,h,\tau}\|_{H^1}^2 \,ds
 \\
& \qquad \qquad   + C_2' \int_{t^{n-1}}^{t^n}  \E_\alpha^\phi(\phi_{h,\tau}|\hat\phi_{h,\tau})\; ds.
\end{align*}
%
The third term can be expended as 
\begin{align*}
(iii) 
&= \la \phi_{h,\tau}(\bu_{h,\tau}-\hbu_{h,\tau}),\nabla(\bmu_{h,\tau}-\hbmu_{h,\tau}) \ra^n + 
\la \phi_{h,\tau}(\bu_{h,\tau}-\hbu_{h,\tau}) \bar r_{2,h,\tau})\ra^n \\
&\le (iiia) + \int_{t^{n-1}}^{t^n} \|\bar \phi_{h,\tau}\|_{H^1} \|\bar \uu - \hbu_{h,\tau}\|_{L^2} \|\bar r_{2,h,\tau}\|_{H^1} \, ds,
\end{align*}
where the term $(iiia)$ involving $\mu$ is kept for later.
Using the uniform $L^\infty(H^1)$ bounds for $\phi_{h,\tau}$ and Young's inequality, we arrive at 
\begin{align*}
(iii) & \le (iiia) + C_3' \int_{t^{n-1}}^{t^n}  \E_\alpha^\uu(\uu_{h,\tau}|\hat\uu_{h,\tau})\; ds
+ C_3 \int_{t^{n-1}}^{t^n} \|\bar r_{2,h,\tau}\|^2_{H^1} \,ds.
\end{align*}
%
The fourth and fifth term can be estimated with similar arguments by
\begin{align*}
(iv) + (v) 
\leq \frac{1}{8} \int_{t^{n-1}}^{t^{n}}  &\D^\mu_{\bar \phi_{h,\tau}}(\bmu_{h,\tau}-\hbmu_{h,\tau}) \, ds \\
&+ C_4 \int_{t^{n-1}}^{t^{n}} \|\bar r_{1,h,\tau}\|_{H^{-1}}^2 \,ds 
+ c_4 \int_{t^{n-1}}^{t^n}  \E_\alpha^\phi(\phi_{h,\tau}|\hat\phi_{h,\tau})\; ds.
\end{align*}
The sixth therm can be estimated using H\"older and embedding inequalities by
\begin{align*}
(vi) &\leq 
\alpha \int_{t^{n-1}}^{t^{n}}  \|\phi_{h,\tau}\|_{L^6} \|\bu_{h,\tau}-\hbu_{h,\tau} \|_{L^3}\|\nabla \phi_{h,\tau} - \bar \phi_{h,\tau}\|_{L^2} \,ds \\
&\leq \frac{1}{8} \int_{t^{n-1}}^{t^{n}} \D^\uu_{\bar \phi_{h,\tau}}(\bu_{h,\tau} - \hbu_{h,\tau}) \,ds   +  C_6' \int_{t^{n-1}}^{t^n}  \E_\alpha(\phi_{h,\tau},\uu_{h,\tau}|\hat\phi_{h,\tau},\hat\uu_{h,\tau})\; ds.
\end{align*}
Here $\D_\phi^\uu(\uu)=\D_{\phi}(0,\uu)$ is the $\uu$-dependent part of the dissipation functional.
Using similar arguments as in the estimate for (ii), the last term can be bounded by
\begin{align*}
(vii) \leq C_7'\int_{t^{n-1}}^{t^n}  \E_\alpha^\phi(\phi_{h,\tau}|\hat\phi_{h,\tau})\; ds.
\end{align*}
The constant $C_7'$ depends on the uniform bounds for $\|\phi_{h,\tau}\|_{L^\infty(H^1)}$,  $\|\hat\phi_{h,\tau}\|_{L^\infty(H^1)}$, and $\norm{\dt\hat\phi_{h,\tau}}_{L^\infty(L^2)}$, 
which are valid under our assumptions.

\subsection*{Second term in $\eqref{eq:relenest}$}
Applying the identity \eqref{eq:aux} to the respective integrand, and using  the identities
\eqref{eq:pg3} and \eqref{eq:disc_pert3}
with admissible test function $\bar\vv_{h,\tau}=\bu_{h,\tau}-\hbu_{h,\tau} \in P_0(\Itau;\Vh^d)$, leads to
\begin{align*}
(II) &= \la \bu_{h,\tau} - \hbu_{h,\tau},\dt \uu_{h,\tau} - \dt \hat\uu_{h,\tau}  \ra^n 
\\
&= -\la \eta(\bar\phi_{h,\tau})\nabla(\bu_{h,\tau}-\hbu_{h,\tau}),\nabla(\bu_{h,\tau} - \hbu_{h,\tau}) \ra^n 
\\
&\qquad \qquad - \la \uu_{h,\tau} \cdot \nabla (\bu_{h,\tau} - \hbu_{h,\tau}),\bu_{h,\tau} - \hbu_{h,\tau} \ra_\skw^n 
+ \la \bar p_{h,\tau} - \hat{\bar{p}}_{h,\tau},\div(\bu_{h,\tau} - \hbu_{h,\tau}) \ra^n
\\
& \qquad \qquad - \la \phi_{h,\tau}(\bu_{h,\tau} - \hbu_{h,\tau}),\nabla(\bmu_{h,\tau} - \hbmu_{h,\tau})\ra + \la \bar \rr_{3,h,\tau},\bu_{h,\tau} - \hbu_{h,\tau} \ra \,  ds\\
 &= -\int_{t^{n-1}}^{t^n} \D^\uu_{\bar \phi_{h,\tau}}(\bu_{h,\tau} - \hbu_{h,\tau}) \,ds + (a) + (b) + (c) + (d).
\end{align*}
The first term is the dissipation caused by viscous effects. %
Due to skew-symmetry of the convective term, we have $(a)=0$. 
By testing the identities \eqref{eq:pg4} and \eqref{eq:disc_pert4} with $\bar q_{h,\tau} = p_{h,\tau}-\hbp_{h,\tau}\in P_0(\Itau;\Qh)$, one can see that $\bu_{h,\tau} - \hbu_{h,\tau} \in P_0(\Itau;\Zh)$ is discretely divergence free, and hence $(b)=0$ as well. 
A brief inspection shows that $(c)=-(iiia)$ from above, and therefore will cancel when summing up the terms later on. 
For the last term, we use 
\begin{align*}
(d) &\le \int_{t^{n-1}}^{t^n} \|\bar \rr_{3,h,\tau}\|_{\Zh^*} \|\bu_{h,\tau} - \hbu_{h,\tau}\|_{H^1} \,ds 
\le \frac{1}{8} \int_{t^{n-1}}^{t^n} \D_{\bar \phi_{h,\tau}}(\bu_{h,\tau} - \hbu_{h,\tau}) \,ds \\ 
& \qquad \qquad + 
c_d \int_{t^{n-1}}^{t^n}  \E_\alpha^\uu(\uu_{h,\tau}|\hat\uu_{h,\tau})\; ds
 + C_d \int_{t^{n-1}}^{t^n} \|\bar \rr_{3,h,\tau}\|_{\Zh^*}^2 \,ds.
\end{align*}
\subsection*{Completion of the proof}
By adding up the detailed estimates for the individual contributions to the terms $(I)$ and $(II)$ in \eqref{eq:relenest}, we obtain the desired estimate.
Let us note that the constants $C_x$ in the above estimates only depend on the bounds for the coefficients, while the constants $C_x'$ may also depend on uniform bounds for the solutions, which are valid under our assumptions. 
\qed

\section{Proof of Lemma~\ref{lem:residual} and estimates for $\mu$}
\label{app:residual}
%

Using assumptions (A0)--(A7), we first establish the bounds
\begin{align*}
&\int_{t^{n-1}}^{t^n} \|\bar r_{1,h,\tau}\|^2_{H^{-1}} 
+ \|\bar r_{2,h,\tau}\|^2_{H^1} + \|\bar \rr_{3,h,\tau}\|^2_{\Zh^*} \, ds \\ 
& \qquad \qquad \le 
\hat c \int_{t^{n-1}}^{t^n}\E_\alpha(\phi_{h,\tau},\uu_{h,\tau}|\hat \phi_{h,\tau},\hat\uu_{h,\tau}) \; ds \\
& \qquad \qquad \qquad \qquad + \frac{1}{4 \bar C} \int_{t^{n-1}}^{t^n} \D_{\bar \phi_{h,\tau}}(\bmu_{h,\tau}-\hbmu_{h,\tau},  \bu_{h,\tau}-\hbu_{h,\tau}) \, ds + \hat C (h^4 + \tau^4),
\end{align*}
for the discrete residuals.
For ease of notation, we abbreviate  $W^{s,q}(t^{n-1},t^n;W^{k,p})$ by $W^{s,q}(W^{k,p})$ in this section, and we consider the three terms separately.

\subsection*{First residual}

From the bounds for the $H^1$-projection error, we immediately obtain
\begin{align*}
\int_{t^{n-1}}^{t^n} \|\bar r_{1,h,\tau}\|^2_{H^{-1}} \,ds
&\le 3\int_{t^{n-1}}^{t^n}  \|\dt (\pi_h^1 \phi - \phi)\|^2_{H^{-1}} +  \|b(\bar\phi_{h,\tau}) \nabla \hbmu_{h,\tau} - \overline{b(\phi) \nabla \mu}\|^2_{L^2}  \\
    &\qquad \qquad \qquad  +\|\hbu_{h,\tau}\bar\phi_{h,\tau}-\oobar{\uu\phi}\|^2_{L^2} \,ds \\
    &\le C h^2 \|\dt \phi\|_{L^2(H^1)}^2 + 3 (*)_1^2 + 3 (*)_2^2,
\end{align*}
where we again used $\overline{a} = \bar \pi_\tau^0 a$ here to abbreviate the projection onto piecewise constant functions in time. 
The remaining terms can be further estimated  as follows:
\begin{align*}
(*)_1^2 
&\le \int_{t^{n-1}}^{t^n} \|b(\bar\phi_{h,\tau}) \nabla (\hbmu_{h,\tau} - \overline{\mu})\|_{L^2}^2 + \|b(\bar\phi_{h,\tau}) - b(\bar\phi)) \nabla \bar \mu\|_{L^2}^2 \\
& \qquad \qquad \qquad  + \|(b(\bar\phi) - \overline{b(\phi)}) \nabla \bar \mu\|_{L^2}^2  + \|\overline{b(\phi) \nabla \mu} - \overline{b(\phi)}\;\overline{\nabla\mu}\|_{L^2}^2 \,ds \\
&= (i)+(ii)+(iii)+(iv).
\end{align*}
Using the bounds of assumption (A2), the definition of $\hbmu_{h,\tau}$, and the stability and error estimates for the $L^2$-projection \eqref{eq:l2projest}, we immediately obtain
\begin{align*}
(i) 
\le C \int_{t^{n-1}}^{t^n} \|\nabla (\pi_h^0 \mu - \mu)\|_{L^2}^2 \,ds 
\le \tilde C h^4 \|\mu\|_{L^2(H^3)}^2. 
\end{align*}
By H\"older inequalities, assumption (A2), and elementary calculus, we obtain 
\begin{align*}
(ii) 
&\le C \int_{t^{n-1}}^{t^n} \|b(\bar \phi_{h,\tau}) - b(\bar \phi)\|_{L^6}^2 \|\mu\|_{W^{1,3}}^2 \,ds \\
&\le \tilde C \|\mu\|_{L^\infty(W^{1,3}_p)} \int_{t^{n-1}}^{t^n} \|\phi_{h,\tau} - \hat \phi_{h,\tau}\|_{L^6}^2 + \|\hat \phi_{h,\tau} -  \phi\|_{L^6}^2\,ds.
\end{align*}
By Sobolev embedding, bounds for the projection and interpolation errors, and the $H^1$-bounds for the relative energy, we see that 
\begin{align*}
(ii) &\le  \tilde C \|\mu\|_{L^\infty(W^{1,3})} \left( h^4  \|\phi\|_{L^2(H^3)}^2 + \tau^4 \|\dtt \phi\|_{L^2(H^1)}^2  + c\int_{t^{n-1}}^{t^n} \E_\alpha^\phi(\phi_{h,\tau}|\hat\phi_{h,\tau}) \; ds \right).
\end{align*}
To estimate (iii), we observe that both averages are second order approximation for the value at the midpoint of the time interval. Using \eqref{eq:midpoint_order_single} thus yields
\begin{align*}
(iii) 
&\leq C \|\mu\|_{L^\infty(W^{1,3})}^2 \int_{t^{n-1}}^{t^n} \|b(\bar\phi)-\overline{b(\phi)}\|_{L^6}^2 \,ds \\
&\leq  \tilde C \tau^4 \|\mu\|_{L^\infty(W^{1,3})}^2 \int_{t^{n-1}}^{t^n} \|\dtt b(\phi)\|_{L^6}^2 \,ds.
\end{align*}
By Sobolev embedding, the chain rule, the triangle inequality, and assumption (A2), we can bound
\begin{align*}
\int_{t^{n-1}}^{t^n} \|\dtt b(\phi)\|_{L^6} \,ds
&\le 2 \int_{t^{n-1}}^{t^n} b_3\|(\dt \phi)^2\|_{H^1}^2 + b_4 \|\dtt \phi\|_{H^1}^2 \,ds \\ 
&\le C \, \left( \|\dt \phi\|_{L^\infty(H^1)}^2 \|\dt \phi\|_{L^2(H^1)}^2 +  \|\dtt \phi\|_{L^2(H^1)}^2 \right).   
\end{align*}
This allows to estimate the term (iii) by powers of $\tau$ and norms of the solution appearing in assumption (A7). 
For the fourth term, we again use \eqref{eq:midpoint_order} to obtain
\begin{align*}
(iv) 
&\le C\tau^4 \| b(\phi)\nabla\mu\|_{H^2(L^2)}^2 
\le C' \tau^4 \left(\|\mu\|_{H^2(H^1)}^2 + \|\phi\|_{H^2(H^1)}^2 \right).
\end{align*}
The constant $C'$ here may depend on bounds for $\mu$ and $\dt \phi$ in the $L^\infty(W^{1,3})$ and $L^\infty(H^1)$ norm, respectively. 
By summing up the individual contributions, we see that 
\begin{align*}
C(*)_1 & \leq C'  h^4 (\|\mu\|_{L^2(H^3)}^2 + \|\phi\|_{L^2(H^3)}^2) + C'' \tau^4 (\|\phi\|_{H^2(H^1)}^2 + \|\mu\|_{H^2(H^1)}^2) \\
& \qquad \qquad + c' \int_{t^{n-1}}^{t^n} \E_\alpha^\phi(\phi_{h,\tau}|\hat\phi_{h,\tau}) \;ds
\end{align*}
with constants $C'$, $C''$ and $c'$ depending on uniform bounds for $\mu$ and $\dt \phi$ in $L^\infty(W^{1,3})$ and $L^\infty(H^1)$, respectively. 
In order to estimate the term $(*)_2$, 
we note that
\begin{align*}
\int_{t^{n-1}}^{t^n} (*)_2^2 \,ds 
&\le  \int_{t^{n-1}}^{t^n}  \norm{\hbu_{h,\tau}\bar\phi_{h,\tau}-\hbu_{h,\tau}\bar\phi}_{L^2}^2 + \norm{\hbu_{h,\tau}\bar\phi-\bar\uu\bar\phi}_{L^2}^2 + \norm{\bar\uu\bar\phi-\oobar{\uu\phi}}_{L^2}^2 \,ds\\
& = (a) + (b) + (c).
\end{align*}
The first term can be estimated by 
\begin{align*}
(a) 
&\le \int_{t^{n-1}}^{t^n} \|\hbu_{h,\tau}\|_{L^3}^2 \|\bar\phi_{h,\tau}-\bar\phi\|_{L^6}^2 \,ds
\\
&\le C \|\uu\|_{L^\infty(H^1)} \left( h^4 \|\phi\|_{L^2(H^3)}^2 + \tau^4 \|\phi\|_{H^2(H^1)}^2 + c \int_{t^{n-1}}^{t^n}\E_\alpha^\phi(\phi_{h,\tau}|\hat\phi_{h,\tau}) \; ds \right).
\end{align*}
For the second term, we obtain in a similar manner
\begin{align*}
(b) 
&\le  \int_{t^{n-1}}^{t^n} \|\phi\|_{L^6}^2 \|\hbu_{h,\tau}-\bar\uu\|_{L^3}^2 \,ds  
\le C \|\phi\|_{L^\infty(H^1)}^2 \left( h^4 \|\uu\|_{L^2(H^3)}^2 + \tau^4 \|\uu\|_{H^2(H^1)}^2 \right).
\end{align*}
For the third term, we again use \eqref{eq:midpoint_order} to obtain
\begin{align*}
(iii) 
&\le C \tau^4 \|\phi\uu\|_{H^2(L^2)}^2  
\le C' \tau^4 \left( \|\phi\|_{H^2(L^4)}^2 + \|\uu\|_{H^2(L^4)}^2 \right),
\end{align*}
with a constant $C'$ that depends on the uniform bounds for norms of $\uu$ and $\phi$ appearing in assumption (A7).
By summing up the individual estimates, we arrive at
\begin{align*}
C(*)_1 
&\le C' h^4 (\|\phi\|_{L^2(H^3)}^2 + \|\uu\|_{L^2(H^3)}^2) + C'' \tau^4 (\|\phi\|_{H^2(H^1)}^2 + \|\uu\|_{H^2(H^1)}^2)  \\
& \qquad \qquad \qquad + c' \int_{t^{n-1}}^{t^n}\E_\alpha^\phi(\phi_{h,\tau}|\hat\phi_{h,\tau})\; ds
\end{align*}
with constants $C'$, $C''$ and $c'$ depending on uniform bounds of the solution  in norms  appear in assumption (A7).
By combination of the previous estimates, we may then bound the first residual as required. 
\qed

\subsection*{Second residual}

Let us start with the observation that
\begin{align} \label{eq:aux2}
\la \nabla (\hat \phi_{h,\tau} - \I_\tau^1 \phi), \nabla \bar \xi_{h,\tau} \ra^n
&= \la \I_\tau^1 \phi - \hat \phi_{h,\tau}, \bar \xi_{h,\tau}\ra^n,
\end{align}
which follows immediately from the definition of $\hat \phi_{h,\tau}$ and \eqref{eq:defh1proj}.
The second residual can then be expressed equivalently as 
\begin{align*}
    \bar r_{2,h,\tau} = (\overline{\pi_h^0 \mu} - \overline{\I_\tau^1 \pi_h^0 \mu}) + (\overline{\I_\tau^1 \phi} - \overline{\hat \phi_{h,\tau}}) + (\overline{f'(\hat \phi_{h,\tau})} - \overline{\I_\tau^1 f'(\phi)}).
\end{align*}
Recall that we use $\overline{g} = \bar \pi_\tau^0 g$ to denote the piecewise constant projection of $g$ with respect to time. 
This pointwise representation allows us to estimate
\begin{align*}
\int_{t^{n-1}}^{t^n} \|\bar r_{2,h,\tau}\|_{H 1}^2 \,ds &\leq \|\pi_h^0 \mu - \I_\tau^1 \pi_h^0 \mu\|^2_{L^2(H^1)} + \|\I_\tau^1 \phi - \hat \phi_{h,\tau}\|^2_{L^2(H^1)} \\
& \qquad \qquad \qquad + \|f'(\hat \phi_{h,\tau}) - \I_\tau^1 f'(\phi)\|^2_{L^2(H^1)}  
= (i) + (ii) + (iii) .   
\end{align*}
We again bound the individual terms separately. 
For the first term, we use the $H^1$-stability of the $L^2$-projection $\pi_h^0$ and the interpolation error estimate \eqref{eq:timinterpest} to obtain 
\begin{align*}
    (i) \le C \|\mu - \I_\tau^1 \mu\|_{L^2(H^1)}^2 \le \tilde C \tau^4 \|\mu\|_{H^2(H^1)}^2.
\end{align*}
For the second term, we employ a triangle inequality and standard interpolation and projection error estimates, to get
\begin{align*}
(ii) 
&\le 3 \left(\|\I_\tau^1 \phi - \phi\|_{L^2(H^1)}^2 + \|\phi - \pi_h^1\phi\|^2_{L^2(H^1)} + \|\pi_h^1 \phi - \I_\tau^1 \pi_h^1 \phi\|_{L^2(H^1)}^2 \right) \\
&\le C  \left( h^4 \|\phi\|_{L^2(H^3)}^2 + \tau^4 \|\phi\|_{H^2(H^1)}^2 \right).
\end{align*}
From assumption (A7), we conclude that 
$\phi$ and its projection $\hat\phi_{h,\tau} = \I_\tau^1 \phi_h^1 \phi$ can be bounded uniformly in $L^\infty(W^{1,\infty})$.
Therefore, all terms $f^{(k)}(\cdot)$ appearing in the following can be estimated in modulus by a constant $C_f$.
This allows us to bound
\begin{align*}
(iii) 
&\le \|f'(\hat \phi_{h,\tau}) - f'(\phi)\|_{L^2(H^1)}^2 + \|f'(\phi) - \I_\tau^1 f'(\phi)\|_{L^2(H^1)}^2 
\\
&\le C_f \|\hat\phi_{h,\tau} - \phi\|_{L^2(H^1)}^2 + \tau^4\|f'(\phi)\|_{H^2(H^1)}^2 
\\ 
&\le \tilde C_f \left( h^4 \|\phi\|_{L^2(H^3)}^2 + \tau^4 \|\phi\|_{H^2(H^1)}^2 + \tau^4 \|\phi\|_{L^2(H^3)}^2 + \tau^5 \right).
\end{align*}
In summary, the second residual may thus be estimated by
\begin{align*}
\int_{t^{n-1}}^{t^n} \|\bar r_{2,h,\tau}(s)\|_{H^1}^2 \,ds 
&\le C' h^4 \|\phi\|_{L^2(H^3)}^2  \\
& \qquad  + C'' \tau^4 \left(  \|\mu\|_{H^2(H^1)}^2 + \|\phi\|_{H^2(H^1)}^2 + \|\phi\|_{L^2(H^3)}^2 + \tau \right), 
\end{align*}
with constants $C'$, $C''$ depending on the uniform bounds of the solution in the norms appearing in assumption (A7). 

\subsection*{Third residual}

By triangle, binomial and H\"older inequalities, we get 
\begin{align*}
\int_{t^{n-1}}^{t^n} \norm{\rr_{3,h,\tau}}_{\Zh^*}^2 
&\le 6 \int_{t^{n-1}}^{t^n} \|\dt(\mathbf{P}^1_h\uu-\uu)\|_{H^{-1}}^2 + \|\eta(\bar\phi_{h,\tau}) \nabla\hbu_{h,\tau} - \overline{\eta(\phi)\nabla\uu}\|_{L^2}^2 \\ 
& \qquad \qquad + \|\hbp_{h,\tau} - \bar p\|_{L^2}^2 + \|\bu_{h,\tau}\hbu_{h,\tau} - \overline{\uu\uu}\|_{L^2}^2 \\
& \qquad \qquad + \|\bu_{h,\tau}\nabla\hbu_{h,\tau} - \overline{\uu\nabla\uu}\|_{L^{6/5}}^2 + \|\bar\phi_{h,\tau} \nabla\hbmu_{h,\tau}-\overline{\phi\nabla\mu}\|_{L^2}^2 \,ds\\
&= (i) + (ii) + (iii) + (iv) + (v) +(vi).
\end{align*}
By standard estimates for the Stokes projection, see \eqref{eq:stokesproj}, the first term is bounded by
\begin{equation*}
(i) \leq Ch^4 \|\dt\uu\|^2_{L^2(H^1)}. 
\end{equation*}
For the second term, we use
\begin{align*}
 (ii) &\leq \tilde C \int_{t^{n-1}}^{t^n}  \|\eta(\bar\phi_{h,\tau}) \nabla(\hbu_{h,\tau} -\bar\uu)\|_{L^2}^2 +  \|(\eta(\bar\phi_{h,\tau})-\overline{\eta(\bar\phi)})\nabla\bar\uu\|_{L^2}^2 \\
 &\qquad \qquad  + \|(\eta(\bar\phi)-\overline{\eta(\phi)})\nabla\bar\uu\|_{L^2}^2 + \|\overline{\eta(\phi)\nabla\uu}-\overline{\eta(\phi)}\nabla\bar\uu\|_{L^2}^2 \,ds\\
 &= (*)_1 + (*)_2 +(*)_3 +(*)_4
\end{align*}
With the bounds for $\eta(\cdot)$ in assumption (A3), the definition of $\hbu_{h,\tau}$, and the standard projection and interpolation estimates of Lemma~\ref{lem:projerr}, we immediately obtain
\begin{align*}
(*)_1 
\le \hat C \, \left(h^4 \|\uu\|_{L^2(H^3)}^2 + \tau^4\|\uu\|_{H^2(H^1)}^2 \right). 
\end{align*}
Using the bounds for $\eta(\cdot)$ and $\uu$ from assumptions (A3) and (A7), we further see that
\begin{align*}
(*)_2
&\le \tilde C\|\phi_{h,\tau}-\phi\|^2_{L^2(L^6)} \|\uu\|_{L^\infty(W^{1,3})}^2 \\
 &\leq C' \left( h^4 \|\phi\|^2_{L^2(H^3)} +\tau^4 \|\phi\|^2_{H^2(H^1)} \right)  + c' \int_{t^{n-1}}^{t^n} \E_\alpha^\phi(\phi_{h,\tau},|\hat \phi_{h,\tau}) \;ds.
\end{align*}
For the third term we again employ \eqref{eq:midpoint_order_single}, and obtain
\begin{align*}
(*)_3 
&\le \tilde C \|\nabla\uu\|_{L^\infty(W^{1,3})}^2 \|\overline{\eta(\bar\phi) - \eta(\phi)}\|^2_{L^2(L^6)} 
\le \tilde C' \tau^4 \|\eta(\phi)\|^2_{H^2(H^1)}\\
&\leq C' \tau^4 \left(\tau + \|\phi\|_{H^1(H^3)}^2 + \|\phi\|_{H^2(H^1)}^2\right).
\end{align*}
For the fourth term, we use \eqref{eq:midpoint_order} to verify that
\begin{align*}
(*)_4 
&\le \tilde C\tau^4 \|\eta(\phi)\nabla\uu\|_{H^2(L^2)}^2 
\le C' \tau^4 (\|\uu\|_{H^2(H^1)}^2 + \|\phi\|_{H^2(H^1)}^2).
\end{align*}
The constant $C'$ again depends on uniform bounds for the solution guaranteed by assumption (A7).
By combination of the previous estimates, we thus see that
\begin{align*}
(ii) &\le C' h^4 \left( \|\phi\|_{L^2(H^3)}^2 + \|\uu\|_{L^2(H^3)}^2 \right) \\
& \qquad \qquad + C'' \tau^4 \left( \tau + \|\phi\|_{H^1(H^3)}^2 + \|\phi\|_{H^2(H^1)}^2 + \|\uu\|_{H^2(H^1)}^2\right) \\
& \qquad \qquad + c'\int_{t^{n-1}}^{t^n} \E_\alpha^\phi(\phi_{h,\tau}|\hat \phi_{h,\tau}) \;ds.
\end{align*}
The constant $C'$, $C''$ and $c'$ again depend only the bounds for the coefficients and the solution in the assumptions. 
Via standard projection error estimates, the third term in the expansion of the residual $\bar \rr_{3,h,\tau}$, can be bounded by
\begin{align*}
(iii) \le C \left( h^4 \|p\|_{L^2(H^2)}^2 + \tau^4 \|p\|_{H^2(L^2)}\right).
\end{align*}
Using triangle and binomial inequalities, and proceeding with similar arguments as already employed in previous estimates, the fourth term can be bounded by
\begin{align*}
(iv) &\leq 3\int_{t^{n-1}}^{t^n} \|\bu_{h,\tau}-\bar\uu\|_{L^6}^2 \|\hbu_{h,\tau}\|^2_{L^3} + \norm{\hbu_{h,\tau}-\bar\uu}_{0,6}^2\norm{\hbu}^2_{L^3}  + \norm{\bar\uu\bar\uu-\oobar{\uu\uu}}_{0}^2 \,ds \\
&\leq  C' h^4 \|\uu\|_{L^2(H^3)}^2 + C'' \tau^4 \|\uu\|_{H^2(H^1)}^2 + \frac{1}{8} \int_{t^{n-1}}^{t^n} \D^\uu_{\bar \phi_{h,\tau}}(\bu_{h,\tau}-\hbu_{h,\tau}) \,ds \\
&\qquad \qquad + c'\int_{t^{n-1}}^{t^n} \E_\alpha^\uu(\uu_{h,\tau}|\hat\uu_{h,\tau}) \; ds.
\end{align*}
The constants $C'$, $C''$ and $c'$ again depend on uniform bounds for the solutions in norms that can be controlled by assumption (A7).
By H\"older inequalities and proceeding with similar arguments as above, the fifth term can be estimated via
\begin{align*}
(v) &\leq \tilde C \int_{t^{n-1}}^{t^n} \|\bu_{h,\tau}-\bar\uu\|_{L^3}^2 \|\nabla\hbu_{h,\tau}\|^2_{L^2} + \|\nabla(\hbu_{h,\tau}-\bar\uu)\|_{L^2}^2 \|\bar\uu\|^2_{L^3}  + \|\bar\uu\nabla\bar\uu-\overline{\uu\nabla\uu}\|_{L^{6/5}}^2 \,ds\\
& \leq C' h^4 + C'' \tau^4 + \frac{1}{8} \int_{t^{n-1}}^{t^n} \D^\uu_{\bar\phi_{h,\tau}}(\bu_{h,\tau}-\hbu_{h,\tau}) \,ds  + c'\int_{t^{n-1}}^{t^n} \E_\alpha^\uu(\uu_{h,\tau}|\hat\uu_{h,\tau}) \; ds.
\end{align*}
The last term in the expansion of the residual can finally be treated by
\begin{align*}
(vi) &\leq 3 \int_{t^{n-1}}^{t^n}  \|\bar \phi \nabla\bar \mu -\overline{\phi\nabla\mu}\|^2_{L^2}
+ \|\bar\phi_{h,\tau}-\bar\phi\|^2_{L^6}
\|\nabla\hbmu_{h,\tau}\|_{L^3}^2 
+ \|\bar\phi\|^2_{L^\infty} \|\nabla (\hbmu_{h,\tau}-\bar\mu)\|_{L^2}^2  
\,ds
\\
&\le C \tau^4 \|\phi\nabla\mu\|_{H^2(L^2)}^2
+ \tilde C \|\mu\|_{L^\infty(W^{1,3})}^2 (h^4 \|\phi\|_{L^2(H^3)}^2 +\tau^4 \|\phi\|_{H^2(H^1)}^2)  \\
& + \hat C \|\phi\|_{L^\infty(L^\infty)}^2 (h^4 \|\mu\|_{L^2(H^3)}^2 +\tau^4 \|\mu\|_{H^2(H^1)}^2)  + c'\int_{t^{n-1}}^{t^n} \E_\alpha^\phi(\phi_{h,\tau}|\hat \phi_{h,\tau}) \; ds.
\end{align*}
The first term in this estimate can be further expanded and estimated by norms appearing also in the other terms. 
In summary, we thus obtain
\begin{align*}
\int_{t^{n-1}}^{t^n} \|\bar \rr_{3,h,\tau}\|^2_{\Zh^*} \,ds
\le \;& C' h^4 (\|\phi\|_{L^2(H^3)}^2 + \|\uu\|_{H^1(H^1)}^2 + \|\uu\|_{L^2(H^3)}^2 + \|p\|_{L^2(H^2)}^2) \\
& + C'' \tau^4 (\tau + \|\phi\|_{H^1(H^3)}^2 + \|\phi\|_{H^2(H^1)}^2 + \|\uu\|_{H^2(H^1)}^2 + \|p\|_{H^2(L^2)}^2) \\
&+ \frac{1}{8} \int_{t^{n-1}}^{t^n}\D^\uu_{\bar \phi_{h,\tau}}(\bu_{h,\tau}-\hbu_{h,\tau}) \,ds  + c'  \int_{t^{n-1}}^{t^n} \E_\alpha(\phi_{h,\tau},\uu_{h,\tau}|\hat \phi_{h,\tau},\hat\uu_{h,\tau}) \; ds.
\end{align*}
The constants $C'$, $C''$ and $c'$ again only depend on bounds for the parameters and uniform bounds for the solution, which are valid under our assumptions.
\qed

\subsection*{Completion of the proof}
The assertion of Lemma~\ref{lem:residual} now follows by simply adding up the individual estimates for the three residuals. \qed

\subsection*{Estimates for $\|\bar\mu_{h,\tau}-\hbmu_{h,\tau}\|_{L^2}$}
\label{subsec:mu}

From the Poincar\'e lemma, the estimate~\eqref{eq:meanmuerr}, Lemma~\ref{lem:residual}, and the bounds for the energy and dissipation functionals, we deduce that
\begin{align*}
\|\bmu_{h,\tau} -  \hbmu_{h,\tau}\|_{L^2(L^2)}^2 
&\le C_P \left( \|\nabla (\bmu_{h,\tau} -  \hbmu_{h,\tau})\|_{L^2(L^2)}^2 + \int_{t^{n-1}}^{t^n} \Big|\la \bmu_{h,\tau} -  \hbmu_{h,\tau},1\ra \Big|^2\, ds  \right)\\
&\le C \int_{t^n}^{t^{n-1}} \D_{\bar \phi_{h,\tau}}^\mu(\bmu_{h,\tau} - \hbmu_{h,\tau}) + \E_\alpha^\phi(\phi_{h,\tau} |\hat \phi_{h,\tau})\, ds + C'(h^4 + \tau^4),
\end{align*}
for every single interval $(t^{n-1},t^n)$. 
With the same arguments as already employed in the proof of Lemma \ref{lem:diskerr}, we then obtain 
$\|\bmu_{h,\tau} - \hbmu_{h,\tau}\|_{L^2(0,T;L^2)}^2 \le C' (h^4 + \tau^4)$.
%

\section{Proof of Lemma~\ref{lem:res2}}
\label{app:residual2}

We now establish the required bounds for the alternative residuals, which appeared in the proof of uniqueness of solutions to the discrete problem. 

\subsection*{First residual}
In straight forward manner, we obtain \begin{align*}
   \int_{t^{n-1}}^{t^{n}}\|\bar r_{1,h,\tau}\|_{H^{-1}}^2 ds &\le  C(\|\hbmu_{h,\tau}\|_{L^\infty(W^{1,3})}^2+\norm{\hbu_{h,\tau}}_{L^\infty(L^{3})}^2)\int_{t^{n-1}}^{t^{n}} \E_\alpha(\phi_{h,\tau},\uu_{h,\tau}|\hat\phi_{h,\tau},\hat\uu_{h,\tau}) ds.\\
\end{align*}
For the third residual we obtain
\begin{align*}
  \int_{t^{n-1}}^{t^{n}}\|\bar \rr_{3,h,\tau}&\|_{\Zh^*}^2 \;ds  \leq C\int_{t^{n-1}}^{t^{n}} \norm{\bu_{h,\tau}-\hbu_{h,\tau}}_{L^2}^2\norm{\nabla\hbu_{h,\tau}}_{L^3}^2 + \norm{\bu_{h,\tau}-\hbu_{h,\tau}}_{L^3}^2\norm{\hbu_{h,\tau}}_{L^6}^2\\
  &\quad+ C\norm{\phi_{h,\tau}-\hat\phi_{h,\tau}}_{L^6}^2\norm{\nabla\hbu_{h,\tau}}_{L^3}^2 + \norm{\nabla\hbmu_{h,\tau}}_{L^{3/2}}^2\norm{\phi_{h,\tau}-\hat\phi_{h,\tau}}_{L^6}^2 ds\\
  & \leq C(\delta)(\norm{\hbu_{h,\tau}}_{L^\infty(W^{1,3})}^2+\norm{\hbmu_{h,\tau}}_{L^\infty(W^{1,3})}^2)\int_{t^{n-1}}^{t^{n}} \E_\alpha(\phi_{h,\tau},\uu_{h,\tau}|\hat\phi_{h,\tau},\hat\uu_{h,\tau}) ds  \\
  & + \int_{t^{n-1}}^{t^{n}} 2\delta\D^\uu_{\bar \phi_{h,\tau}}(\bu_{h,\tau}-\hbu_{h,\tau}) \;ds.
\end{align*}
Combination and choosing delta small enough yields
\begin{align*}
   \int_{t^{n-1}}^{t^{n}}\|\bar r_{1,h,\tau}\|_{H^{-1}}^2 &+\|\bar \rr_{3,h,\tau}\|_{\Zh^*}^2 \;ds \leq  \frac{1}{2}\int_{t^{n-1}}^{t^{n}} \D^\uu_{\bar\phi_{h,\tau}}(\bu_{h,\tau}-\hbu_{h,\tau})\;ds\\
    & +C(\norm{\hbu_{h,\tau}}_{L^\infty(W^{1,3})}+ \norm{\hbmu_{h,\tau}}_{L^\infty(W^{1,3})})\int_0^{t} \E_\alpha(\phi_{h,\tau},\uu_{h,\tau}|\hat\phi_{h,\tau},\hat\uu_{h,\tau}) \;ds. 
\end{align*}
In order to use the discrete Gronwall lemma we need to bound the appearing norms suitably.
Furthermore, we observe that $\bar\uu_{h,\tau}, \bar\mu_{h,\tau}$ are uniformly bounded in $L^\infty(W^{1,3})$. This can be seen as follows.
\begin{equation*}
\|\hat\uu_{h,\tau}\|_{L^\infty(W^{1,3})} \leq \norm{\bar\uu_{h,\tau}-\mathbf{P}^1_h\bar\uu}_{L^\infty(W^{1,3})} + \norm{\mathbf{P}^1_h\bar\uu-\bar\uu}_{L^\infty(W^{1,3})} + \norm{\bar\uu}_{L^\infty(W^{1,3})}.    
\end{equation*}
The last two terms are uniformly bounded by assumption and standard projection errors estimates. The first term can be controlled by employing the inverse inequality \eqref{eq:inverse} with $p=3,q=2,d\leq 3$ in space and $p=\infty,q=2,d=1$ in time, as well as the convergence estimate derived before. This leads to
\begin{equation*}
  \norm{\bar\uu_{h,\tau}-\mathbf{P}^1_h\bar\uu}_{L^\infty(W^{1,3})} \leq C\tau^{-1/2}h^{-1/2}\norm{\bar\uu_{h,\tau}-\mathbf{P}^1_h\bar\uu}_{L^2(H^1)} \leq C_1'h^{-1/2}\tau^{3/2} + C_2'h^{3/2}\tau^{-1/2}.  
\end{equation*}
The same holds true for the other contribution involving $\hbmu_{h,\tau}$ using the $L^2$-projection instead of the Stokes-projection. Hence, we observe that $\tau=c' h$ is a natural choice.
\qed

\section{Proof of Lemma~\ref{lem:time_derivative}}
\label{app:time-derivative}
In this section, we first introduce some auxiliary notation and results, and then establish the proof Lemma~\ref{lem:time_derivative}.
We start by introducing the discrete Stokes operator $\Ah:\Zh \to \Zh$, which is defined by the variational identity
\begin{align*}
   \la \uu_h,\vv_h \ra + \la \nabla\uu_h,\nabla\vv_h \ra  = \la \Ah\uu_h,\vv_h \ra = \la \Ah^{1/2}\uu_h,\Ah^{1/2}\vv_h \ra \qquad \forall \vv_h\in \Zh.
\end{align*}
By construction $\Ah$ is symmetric and positive definite and we observe that 
\begin{equation}
  \norm{\Ah^{1/2}\vv_h}_{L^2} = \norm{\vv_h}_{H^1} \text{ for all } \vv_h\in\Zh.   \label{eq:discstokeseuiv}
\end{equation}
The discrete Leray projection $\Pi_h:L^2(\Omega)^d\to \Zh$, on the other hand, is defined by
\begin{align*}
    \la \uu- \Pi_h \uu,\vv_h \ra = 0 \qquad \forall \vv_h\in\Zh.
\end{align*}
On quasi-uniform meshes, this operator is stable in $H^1(\Omega)$; see \cite{ayuso_refined} for details.
With the help of these operators, we can show the following bound.

\begin{lemma}\label{lem:ayuso1}
For any $\bbf\in L^2(\Omega)^d$, one has
\begin{align*}
   \norm{\Ah^{-1/2}\Pi_h \bbf}_{L^2} \leq \norm{\bbf}_{H^{-1}} .
\end{align*}
\end{lemma}
\begin{proof}
For any $\vv_h\in\Zh$, we can easily see that
\begin{align*}
       \la \bbf,\Ah^{-1/2}\vv_h \ra = \la \Pi_h \bbf,\Ah^{-1/2}\vv_h  \ra = \la \Ah^{-1/2}\Pi_h \bbf,\vv_h  \ra.
\end{align*}
For the special choice $\vv_h=\Ah^{-1/2}\Pi_h \bbf \in \Zh$, we then obtain
\begin{align*}
      \norm{\Ah^{-1/2}\Pi_h \bbf}_{L^2}^2 &= \la \bbf,\Ah^{-1}\Pi_h \bbf \ra \leq \norm{\bbf}_{H^{-1}}\norm{\Ah^{-1}\Pi_h \bbf }_{H^1} = \norm{\bbf}_{H^{-1}}\norm{\Ah^{-1/2}\Pi_h \bbf }_{L^2},
\end{align*}
where we used symmetry of the discrete Stokes operator $\Ah$ and \eqref{eq:discstokeseuiv}.
\end{proof}

As a direct consequence, we then obtain the following estimate.

\begin{lemma}\label{lem:ayuso2}
There existes a constant $C$ independent of $h$, such that
\begin{align*}
    \|\bbf\|_{H^{-1}} \le C \|\Ah^{-1/2}\bbf\|_{L^2} \qquad \forall  \bbf\in\Zh.
\end{align*}    
\end{lemma}
\begin{proof}
We start with the following observation
\begin{align*}
\|\bbf\|_{H^{-1}} 
&= \sup_{\vv \in H^1} \frac{\la \bbf,\vv\ra}{\|\vv\|_{H^1}} 
= \sup_{\vv \in H^1} \frac{\la \bbf,\Pi_h \vv\ra}{\|\vv\|_{H^1}},
\end{align*}  
where we use the $L^2$-orthogonality of the discrete Leray-projection, and that $\bbf\in\Zh$.
Since $\Ah$ is symmetric and positive definite, we can further see that 
\begin{align*}
\|\bbf\|_{H^{-1}}
&= \sup_{\vv \in H^1} \frac{\la \Ah^{-1/2}\bbf,\Ah^{1/2} \Pi_h \vv\ra}{\|\vv\|_{H^1}} \\
&\le \|\Ah^{-1/2} \bbf\|_{L^2} \sup_{\vv \in H^1} \frac{\|\Ah^{1/2} \Pi_h \vv\|_{L^2}}{\|\vv\|_{H^1}}.
\end{align*}
Hence, it remains to show that $\|\Ah^{1/2} \Pi_h \vv\|_{L^2}$ can be estimated appropriately by $\|\vv\|_{H^1}$. 
By construction $\|\Ah^{1/2} \Pi_h \vv\|_{L^2} = \norm{\Pi_h \vv}_{H^1}$, and hence the claim is reduced to the $H^1$-stability of the discrete Leray-projection, which holds on quasi-uniform meshes.
\end{proof}

We can now return to the main goal of this section, i.e., the proof of Lemma~\ref{lem:time_derivative}.
Using Lemma~\ref{lem:ayuso2} for $\bbf=\dt(\uu_h-\hat\uu_h)\in\Zh$, we see that 
\begin{equation*}
 \int_{t^{n-1}}^{t^n}\norm{\dt\uu_{h,\tau} -\dt\hat\uu_{h,\tau}}_{H^{-1}}^2\; ds \leq \int_{t^{n-1}}^{t^n}\norm{\Ah^{-1/2}(\dt\uu_{h,\tau} -\dt\hat\uu_{h,\tau})}_{L^2}^2\; ds
\end{equation*}
For ease of notation, we abbreviate $\be_{h,\tau}=\uu_{h,\tau}-\hat\uu_{h,\tau}$ in the following.
Due to the previous lemma, it suffices to bound $\norm{\Ah^{-1/2}\dt\be_{h,\tau}}_{L^2}$. 
By taking the difference of \eqref{eq:pg3} and \eqref{eq:disc_pert3}, restricted to the space $\Zh$, and using $\vv_h=\Ah^{-1}\dt\be_{h,\tau}\in\Zh$ as a test function, we obtain
\begin{align*}
\int_{t^{n-1}}^{t^n} \|\Ah^{-1/2}\dt\be_{h,\tau}\|_{L^2}^2 
&= - \la \eta(\bar\phi_{h,\tau})\nabla\bar\be_{h,\tau},\nabla \Ah^{-1}\dt\be_{h,\tau}\ra^{n}   - \la \bar\uu_{h,\tau} \cdot \nabla \bar\be_{h,\tau}, \Ah^{-1}\dt\be_{h,\tau} \ra_\skw^{n} \\
& \qquad \qquad - \la \bar\phi_{h,\tau}\nabla(\bar\mu_{h,\tau}-\hbmu_{h,\tau}),\Ah^{-1}\dt\be_{h,\tau}\ra^{n} + \la \bar \rr_{3,{h,\tau}},\Ah^{-1}\dt\be_{h,\tau}\ra^{n}\\
&= (i) + (ii) + (iii) + (iv).
\end{align*}
For the first term, we use Hölder's inequality and \eqref{eq:discstokeseuiv}, which gives
\begin{equation*}
(i) \leq C\int_{t^{n-1}}^{t^n}\norm{\nabla\bar\be_{h,\tau}}_{L^2}\norm{\nabla \Ah^{-1}\dt\mathbf{e}_{h,\tau}}_{L^2} \, ds
\leq C\norm{\nabla\bar\be_{h,\tau}}_{L^2(L^2)}\norm{\Ah^{-1/2}\dt\be_{h,\tau}}_{L^2(L^2)}.     
\end{equation*}
By expansion of the trilinear term $\la a \cdot \nabla b, c \ra_\skw$, we see
\begin{align*}
(ii) = \int_{t^{n-1}}^{t^n} \frac{1}{2}\la (\bar\uu_{h,\tau}\cdot\nabla)\bar\be_{h,\tau}, \Ah^{-1}\dt\be_{h,\tau}\ra - \frac{1}{2}\la \bar\be_{h,\tau},(\bar\uu_{h,\tau}\cdot\nabla) \Ah^{-1}\dt\be_{h,\tau}\ra \, ds.     
\end{align*}
For the first term, we use $\mathbf{I} = \Ah^{-1/2}\Ah^{1/2}$, and the second term can be directly estimated, using \eqref{eq:discstokeseuiv} again. This leads to
\begin{align*}
(ii) \leq C\int_{t^{n-1}}^{t^n}(\norm{\Ah^{-1/2}\Pi_h(\bar\uu_{h,\tau}\cdot\nabla)\bar\be_{h,\tau}}_{L^2} + \norm{\bar\be_{h,\tau}}_{{L^2}}\norm{\bar\uu_{h,\tau}}_{L^\infty})\norm{\Ah^{-1/2}\dt\be_{h,\tau}}_{L^2} \, ds. 
\end{align*}
For the third and fourth terms, we again split $\mathbf{I} = \Ah^{-1/2}\Ah^{1/2}$ and use \eqref{eq:discstokeseuiv}, to verify that
\begin{align*}
(iii) + (iv) 
&\leq \int_{t^{n-1}}^{t^n} \Big(\norm{\Ah^{-1/2}\Pi_h\bar\phi_{h,\tau}\nabla(\bar\mu_{h,\tau}-\hbmu_{h,\tau})}_{L^2} \\ 
& \qquad \qquad \qquad + \norm{\Ah^{-1/2}\Pi_h\bar \rr_{3,{h,\tau}}}_{L^2}\Big)\norm{\Ah^{-1/2}\dt\be_{h,\tau}}_{L^2} \, ds.   
\end{align*}
In summary, this yields the bound
\begin{align*}
\int_{t^{n-1}}^{t^n} 
\|&\Ah^{-1/2}\dt  \be_{h,\tau}\|_{L^2}^2  \, ds
\leq C\int_{t^{n-1}}^{t^n}\norm{\nabla\bar\be_{h,\tau}}_{L^2} +  \|\Ah^{-1/2}\Pi_h(\bar\uu_{h,\tau}\cdot\nabla)\bar\be_{h,\tau}\|_{L^2} \\
& 
+ \norm{\bar\be_{h,\tau}}_{L^2}\norm{\bar\uu_{h,\tau}}_{L^\infty} 
  + \norm{\Ah^{-1/2}\Pi_h\bar\phi_{h,\tau}\nabla(\bar\mu_{h,\tau}-\hbmu_{h,\tau})}_{L^2} 
  + \norm{\Ah^{-1/2}\Pi_h\bar \rr_{3,{h,\tau}}}_{L^2}  \, ds. 
\end{align*}
With the help of Lemma~\ref{lem:ayuso1}, we can estimate $\norm{\Ah^{-1/2}\Pi_h f}_{L^2} \leq \norm{f}_{H^{-1}}$, which implies
\begin{align*}
  \int_{t^{n-1}}^{t^n}\norm{\Ah^{-1/2}&\dt\be_{h,\tau}}_{L^2}^2  \, ds
  \leq C\int_{t^{n-1}}^{t^n}\norm{\nabla\bar\be_{h,\tau}}_{L^2} +  \norm{(\bar\uu_{h,\tau}\cdot\nabla)\bar\be_{h,\tau}}_{H^{-1}} \\
  & \qquad + \norm{\bar\be_{h,\tau}}_{L^2}\norm{\bar\uu_{h,\tau}}_{L^\infty} 
  +  \norm{\bar\phi_{h,\tau}\nabla(\bar\mu_{h,\tau}-\hbmu_{h,\tau})}_{H^{-1}} + \norm{\bar \rr_{3,{h,\tau}}}_{\Zh^*}  \, ds.
\end{align*}
With direct estimates, similar to the proof of Lemma~\ref{lem:perropart1}, we then arrive at
\begin{align*}
\int_0^t \norm{\dt\be_{h,\tau}}_{H^{-1}}^2\, ds \leq C\int_0^t \norm{\Ah^{-1/2}\dt\be_{h,\tau}}_{L^2}^2 \, ds \leq C(h^4+\tau^4).
\end{align*}
The error estimate~\eqref{eq:con_res_2} then follows readily with the triangle inequality and the bounds for the projection error stated Lemma \ref{lem:residual}.
\qed

\subsection*{Remark}
Let us conclude by noting that, in the above proof, we have used the natural restriction $\tau\approx h$ to obtain convergence rates for the pressure.  The pressure is reconstructed via the discrete inf-sup stability \eqref{eq:infsup}, hence there is a unique pressure for every velocity.




\end{document}